\documentclass[11pt,letterpaper,reqno]{amsart}
\usepackage{fullpage}
\usepackage{amsmath,amsthm,amsfonts,amssymb,amscd}
\usepackage{empheq}
\usepackage{lipsum}
\usepackage{lastpage}
\usepackage{thmtools}
\usepackage{enumerate}
\usepackage[shortlabels]{enumitem}
\usepackage{fancyhdr}
\usepackage{mathrsfs}
\usepackage{xcolor}
\usepackage{graphicx}
\usepackage{listings}
\usepackage{bbm}
\usepackage{array}
\usepackage{hyperref}
\usepackage[makeroom]{cancel}
\usepackage[utf8]{inputenc}
\usepackage{comment}
\usepackage{makecell}

\mathtoolsset{showonlyrefs=true}
\numberwithin{figure}{section}
\numberwithin{table}{section}

\hypersetup{%
  colorlinks=true,
  linkcolor=blue,
  citecolor=blue,
  linkbordercolor={0 0 1}
}

\allowdisplaybreaks

\newtheorem{theorem}{Theorem}[section]
\newtheorem{corollary}[theorem]{Corollary}
\newtheorem{lemma}[theorem]{Lemma}
\newtheorem{formula}[theorem]{Formula}
\newtheorem{mytable}[theorem]{Table}
\newtheorem{proposition}[theorem]{Proposition}

\mathtoolsset{showonlyrefs=true}
\numberwithin{equation}{section}

\newcommand{\lrabs}[1]{\left\lvert #1 \right\lvert}
\newcommand{\lrp}[1]{\left(#1\right)}
\newcommand{\lrb}[1]{\left[#1\right]}
\newcommand{\lrfloor}[1]{\left\lfloor #1 \right\rfloor}

\newcommand{\condf}{{\mathfrak{f}}}
\newcommand{\midmid}{\parallel}

\newcommand{\pttmatrix}[4]{
\left(
\begin{smallmatrix}
  #1 & #2 \\
  #3 & #4
\end{smallmatrix}
\right)
}

\newcommand{\nw}{{\textnormal{new}}}

\setlist[enumerate]{leftmargin=*,widest=0}
\setlist[itemize]{leftmargin=*,widest=0}
\makeatletter
\def\subsection{\@startsection{subsection}{2}%
  \z@{.5\linespacing\@plus.7\linespacing}{.3\linespacing}%
  {\normalfont\bfseries}
}

\begin{document}

\title{Newspaces with Nebentypus: An Explicit Dimension Formula and Classification of Trivial Newspaces}


\author[E. Ross]{Erick Ross}
\address[E. Ross]{School of Mathematical and Statistical Sciences, Clemson University, Clemson, SC, 29634}
\email{erickr@clemson.edu}

\begin{abstract}
    Consider $N \geq 1$, $k \geq 2$, and $\chi$ a Dirichlet character modulo $N$ such that $\chi(-1) = (-1)^k$. For any bound $B$, one can show that $\dim S_k(\Gamma_0(N),\chi) \le B$ for only finitely many triples $(N,k,\chi)$. 
    It turns out that this property does not extend to the newspace; there exists an infinite family of triples $(N,k,\chi)$ for which $\dim S_k^{\text{new}}(\Gamma_0(N),\chi) = 0$. 
    However, we classify this case entirely. We also show that excluding the infinite family for which $\dim S_k^{\text{new}}(\Gamma_0(N),\chi) = 0$, $\dim S_k^{\text{new}}(\Gamma_0(N),\chi) \leq B$ for only finitely many triples $(N,k,\chi)$. In order to show these results, we derive an explicit dimension formula for the newspace $S_k^{\text{new}}(\Gamma_0(N),\chi)$. 
    We also use this explicit dimension formula to prove a character equidistribution property and disprove a conjecture from Greg Martin that $\dim S_2^{\text{new}}(\Gamma_0(N))$ takes on all possible non-negative integers.
\end{abstract}

\maketitle

\section{Introduction} \label{sec:intro}
Let $N\geq1$, $k\geq2$, and $\chi$ be a Dirichlet character modulo $N$ such that $\chi(-1) = (-1)^k$. The space of cuspforms of level $N$, weight $k$, and nebentypus $\chi$ is denoted by $S_k(\Gamma_0(N),\chi)$ \cite[Section 7.2]{cohen-stromberg}. When the character is trivial, we will just write $S_k(\Gamma_0(N))$. Now, an explicit formula for the dimension of these spaces has been computed; for example, Formula \ref{form:fullspace-dim} below from \cite{cohen-stromberg}. After some analysis on the terms in Formula \ref{form:fullspace-dim}, one can see that $\dim S_k(\Gamma_0(N),\chi)$ has a main term which grows as a function of $N$, and as a function of $k$.
This means that for any fixed $N$ and $\chi$, one could compute a complete list of $k$ for which $\dim S_k(\Gamma_0(N),\chi)$ is small. For example, in the case of trivial character, this would be straightforward from \cite[Table 7.3]{cohen-stromberg}. Similarly, for any fixed $k$, one can compute all the $N,\chi$ for which $S_k(\Gamma_0(N),\chi)$ is small; this could be done using the techniques of \cite[Proposition 26]{martin}, for example.

In this paper, we give a slightly more general result.
\begin{theorem} \label{thm:fullspace-dim}
    Consider $N \geq 1$, $k \geq 2$, and $\chi$ a Dirichlet character modulo $N$ such that $\chi(-1) = (-1)^k$. Then for all bounds $B \geq 0$, $\dim S_k(\Gamma_0(N), \chi) \leq B$ for only finitely many triples $(N,k,\chi)$.
\end{theorem}
This result is not particularly surprising, but we prove it here for lack of reference in the literature. We also give the complete list of triples $(N,k,\chi)$ for which $\dim S_k(\Gamma_0(N), \chi) \leq 2$.

A natural question one might then ask would be if this same result holds on the newspace: for any bound $B$, will $\dim S_k^\nw(\Gamma_0(N),\chi)$ be $\leq B$ for only finitely many triples $(N,k,\chi)$? It turns out this is not true; there exists an infinite family of $(N,k,\chi)$ for which $\dim S_k^\nw(\Gamma_0(N),\chi) = 0$. However, we classify this case entirely; the following theorem completely settles the problem raised in \cite[Theorem 6.1]{shemanske-treneer-walling}, which showed that for a certain class of $N$, $k$, and $\chi$, $S_k^\nw(\Gamma_0(N),\chi)$ is trivial in only finitely many cases.
\begin{theorem} \label{thm:newspace-trivial}
    Consider $N \geq 1$, $k \geq 2$, and $\chi$ a Dirichlet character modulo $N$ of conductor $f=\condf(\chi)$ with $\chi(-1) = (-1)^k$. Then $\dim S_k^\nw(\Gamma_0(N),\chi) = 0$ precisely for the triples $(N,k,\chi)$ such that $2 \mid f$ and $2 \midmid N/f$, together with the finite list given in Table \ref{table:newspace-dim0}.
\end{theorem}

Excluding the infinite family given in Theorem \ref{thm:newspace-trivial}, we show that $\dim S_k^\nw(\Gamma_0(N),\chi) \leq B$ for only finitely many triples $(N,k,\chi)$.

\begin{theorem} \label{thm:newspace-small}
    Consider $N \geq 1$, $k \geq 2$, and $\chi$ a Dirichlet character modulo $N$ of conductor $f=\condf(\chi)$ with $\chi(-1) = (-1)^k$. Also assume it is not the case that $2 \mid f$ and $2 \midmid N/f$. Then for all bounds $B \geq 0$,  $\dim S_k^\nw(\Gamma_0(N),\chi) \leq B$ for only finitely many triples $(N,k,\chi)$.
\end{theorem}

We also give the complete list of triples $(N,k,\chi)$ for which $\dim S_k^\nw(\Gamma_0(N),\chi) \leq 1$.

In \cite[Conjecture 27]{martin}, Greg Martin conjectured that $\dim S_2^\nw(\Gamma_0(N))$ takes on all possible non-negative integers. It turns out this is not true; $67846$ is the first excluded value. In order to prove this, we take $B=67846$ in Theorem \ref{thm:newspace-small} and compute how large $N$ needs to be in order for $\dim S_2^\nw(\Gamma_0(N))$ to be $> 67846$. Then by computing the finitely many remaining cases by computer, we show that $\dim S_2^\nw(\Gamma_0(N)) = 67846$ for no value of $N$.

Now, in order to prove Theorems \ref{thm:newspace-trivial} and \ref{thm:newspace-small}, we need an explicit dimension formula for $S_k^\nw(\Gamma_0(N),\chi)$. Note that here (and throughout the entire paper), $f=\condf(\chi)$ denotes the conductor of $\chi$. Then it is already well-known that 
\begin{align} \label{eqn:newspace-dim-temp}
    \dim S_k^\nw(\Gamma_0(N),\chi) = \sum_{f \mid M \mid N} \beta(N/M) \dim S_k(\Gamma_0(M), \chi),
\end{align}
where $\beta$ is a certain multiplicative function that takes on positive and negative values \cite[Corollary 13.3.7]{cohen-stromberg}. 
However, a dimension formula of this form is useless for our purposes. We are interested in estimating the size of $\dim S_k^\nw(\Gamma_0(N),\chi)$. And this is not possible using the dimension formula in the above form; it is unclear when the positive and negative terms cancel out.

In the case of trivial character, \eqref{eqn:newspace-dim-temp} resembles a Dirichlet convolution. So by writing ${\dim S_k(\Gamma_0(N))}$ in terms of multiplicative functions, Greg Martin \cite{martin} was able to use the theory of Dirichlet convolutions to compute an explicit formula for $\dim S_k^\nw(\Gamma_0(N))$.
However, in the general character case, \eqref{eqn:newspace-dim-temp} is not a proper Dirichlet convolution; several of the terms are missing. So in this paper, we develop some machinery to deal with these ``partial Dirichlet convolutions". We are then able to use this machinery to compute an explicit dimension formula for $\dim S_k^\nw(\Gamma_0(N),\chi)$. 

\begin{theorem} \label{thm:explicit-dim-formula}
    Let $N\geq1$, $k\geq2$, and $\chi$ be a Dirichlet character modulo $N$ of conductor $f=\condf(\chi)$ with $\chi(-1) = (-1)^k$. Then 
    \begin{align} 
        \dim S_k^\nw(\Gamma_0(N),\chi) &= \frac{k-1}{12} \cdot \psi(f) \cdot \beta {*} \psi_f(N/f)   - c_3(k) \cdot \rho(f) \cdot \beta {*} \rho_f(N/f)  \\
        & \mspace{0mu} - c_4(k) \cdot \rho'(f) \cdot \beta {*} \rho_f'(N/f)  - \frac{1}{2} \cdot 2^{\omega(f)} \cdot \beta{*}\sigma_f(N/f) + c_0(k,\chi) \cdot \mu(N/f),
    \end{align}
    where each of these functions is defined below.
\end{theorem}

Several explicit dimension formulas have recently been proven for various spaces of modular forms 
(e.g. in \cite{kimball}, \cite{choi2}, \cite{zhang-zhou}, \cite{kimball2023}, \cite{roy}, \cite{dalal}), and Theorem \ref{thm:explicit-dim-formula} gives a further result in this direction.

Now, our main motivation for computing this explicit dimension formula is to classify small newspaces. However, from this explicit formula, it is also straightforward to study how different characters affect the dimension of $S_k^\nw(\Gamma_0(N),\chi)$.

For comparison, first consider the full space $S_k(\Gamma_0(N),\chi)$. Then we have the following character equidistribution property.

\begin{theorem} \label{thm:fullspace-char-equid}
    Consider $N \geq 1$, $k \geq 2$, and $\chi$ a Dirichlet character modulo $N$ with $\chi(-1) = (-1)^k$.
    Recall that $S_k(\Gamma_1(N)) = \bigoplus_{\chi} S_k(\Gamma_0(N),\chi)$.
    Then asymptotically in $N$, the terms in this direct sum are equidistributed by character. 
\end{theorem}
When we say that this direct sum is ``asymptotically equidistributed by character", we mean that each term in the direct sum has the same dimension, asymptotically (this is made precise in Section \ref{sec:char-equid}).

One might guess that a similar equidistribution property holds for the newspace $S_k^\nw(\Gamma_1(N)) = \bigoplus_{\chi} S_k^\nw(\Gamma_0(N),\chi)$. 
However, this turns out to not be true; $\dim S_k^\nw(\Gamma_0(N),\chi)$ depends heavily on the conductor of $\chi$.
But we are still able prove an equidistribution property if we only consider characters of a given conductor.

\begin{theorem} \label{thm:newspace-char-equid}
    Consider $N \geq 1$, $k \geq 2$, and $\chi$ a Dirichlet character modulo $N$ with $\chi(-1) = (-1)^k$.
    Recall that $S_k^\nw(\Gamma_1(N)) = \bigoplus_{\chi} S_k^\nw(\Gamma_0(N),\chi)$.
    Then asymptotically in $N$, the terms in this direct sum corresponding to characters of a given conductor are equidistributed by character. 
    Furthermore, if $k \equiv 1 \mod 12$, then for all $N$, the terms corresponding to characters of a given conductor are exactly equidistributed by character. 
\end{theorem}

The paper is organized as follows. In Section \ref{sec:dim-formula-fullspace}, we give a dimension formula for $S_k(\Gamma_0(N),\chi)$. Then after some analysis on the terms in this formula, we prove Theorem \ref{thm:fullspace-dim} and give the complete list of triples $(N,k,\chi)$ for which $\dim S_k(\Gamma_0(N),\chi) \leq 2$. 
Then in Section \ref{sec:dim-formula-newspace}, we turn our attention to the newspace. We state the classical formula for $S_k^\nw(\Gamma_0(N),\chi)$, and then rewrite the dimension formula for $S_k(\Gamma_0(N),\chi)$ as a linear combination of multiplicative functions.
Next in Section \ref{sec:partial-convolutions}, we  develop some machinery to calculate partial Dirichlet convolutions.
In Section \ref{sec:explicit-dim-formula}, we use this machinery to compute each of the partial convolutions appearing in \eqref{eqn:newspace-dim-temp}. This yields an explicit formula for $\dim S_k^\nw(\Gamma_0(N),\chi)$, proving Theorem \ref{thm:explicit-dim-formula}. Then in Section \ref{sec:when-dim-newspace-small}, we use the explicit dimension formula from Section \ref{sec:explicit-dim-formula} to prove Theorems \ref{thm:newspace-trivial} and \ref{thm:newspace-small}. We also give the complete list of triples $(N,k,\chi)$ (excluding the infinite family from Theorem \ref{thm:newspace-trivial}) for which $\dim S_k^\nw(\Gamma_0(N),\chi) \leq 1$. 
Then in Section \ref{sec:disproof-conj}, we use the bounds from Theorem \ref{thm:newspace-trivial} to disprove \cite[Conjecture 27]{martin}.
Next, in Section \ref{sec:char-equid}, we study character distribution and prove Theorems \ref{thm:fullspace-char-equid} and \ref{thm:newspace-char-equid}.
Finally, in Section \ref{sec:decomposition}, we discuss a motivation to better understand why the newspace is trivial when $2 \mid f$ and $2 \midmid N/f$. In particular, we show that in this case, $S_k(\Gamma_0(N),\chi)$ can be decomposed into old parts $ S_k(\Gamma_0(N),\chi) = S_k(\Gamma_0(N/2),\chi) \oplus S_k(\Gamma_0(N/2),\chi) |_k \pttmatrix{2}{0}{0}{1}$.

\section{The dimension of \texorpdfstring{$S_k(\Gamma_0(N),\chi)$}{S\_k(Gamma0(N),chi)} } \label{sec:dim-formula-fullspace}
We start by giving a dimension formula for $S_k(\Gamma_0(N), \chi)$.

\begin{formula}[{\cite[Theorem~7.4.1]{cohen-stromberg}}] \label{form:fullspace-dim}
    Let $N\geq1$, $k\geq2$, and $\chi$ be a Dirichlet character modulo $N$ of conductor $f=\condf(\chi)$ such that $\chi(-1) = (-1)^k$. Then 
    \begin{align}
        \dim S_k(\Gamma_0(N), \chi) \ =& \ \    
        \frac{k-1}{12} \psi(N) 
        \ - \  c_3(k) \mspace{-30mu}\sum_{
            \substack{x \mod N \\  x^2+x+1\equiv0 \mod N}
        }\mspace{-35mu} \chi(x) 
        \ - \ c_4(k) \mspace{-20mu}\sum_{
            \substack{x \mod N \\  x^2+1\equiv0 \mod N}
        }\mspace{-25mu} \chi(x) \\
        &- \frac12 \mspace{-20mu}\sum_{
            \substack{d \mid N \\  (d,N/d) \mid N / f}
        }\mspace{-20mu} \phi(\gcd(d,N/d))
        \ + \ c_0(k,\chi),
    \end{align}
    where
    \begin{align}
        \psi(N) &= N \prod_{p|N} \lrp{ 1 + \frac{1}{p} }, \label{eqn:psi}  \\
        c_3(k) &= \frac{k-1}{3} - \lrfloor{\frac{k}{3}},  \label{eqn:c3}  \\
        c_4(k) &= \frac{k-1}{4} - \lrfloor{\frac{k}{4}},  \label{eqn:c4}   \\
        c_0(k,\chi) &= 
        \begin{cases} 
            1 & \text{if $k=2$ and $\chi$ is trivial,} \\
            0 & \text{otherwise,}
        \end{cases} \label{eqn:delta} \\
        \phi \text{ den}&\text{otes the Euler totient function.} 
    \end{align}
\end{formula}

Next, we prove Lemma \ref{lem:congruence-solns}, which will be used to bound the second and third terms from the above formula. 
\begin{lemma} \label{lem:congruence-solns}
    Let $S(N)$ denote the set of solutions to the congruence $x^2 + x + 1 \equiv 0 \pmod{N}$, and  $S'(N)$ denote the set of solutions to the congruence $x^2 + 1 \equiv 0 \pmod{N}$. 
    Then 
    \begin{equation}
    \begin{alignedat}{4}
        S(3) &= \{1\}, && \\
        S'(2) &= \{1\}, && \\
        S(2^r) &= \varnothing && \text{for } r \geq 1, \\
        S(3^r),\ S'(2^r) &= \varnothing \qquad\qquad && \text{for } r \geq 2, \\
        \#S(N),\ \#S'(N) &\leq 2^{\omega(N)}. && 
    \end{alignedat}
    \end{equation}
    Here, $\omega(N)$ denotes the number of distinct prime divisors of $N$.
\end{lemma}
\begin{proof}
    The first and second claims follow immediately by direct computation. 

    The third and fourth claims follow immediately from the facts that the congruences $x^2 + x + 1 \equiv 0 \pmod{2}$, $x^2 + x + 1 \equiv 0 \pmod{9}$, and $x^2 + 1 \equiv 0 \pmod{4}$ have no solutions.
    
    For the fifth claim, note we already have $\#S(3^r) \leq 1$ for all $r \geq 1$. For primes $p \neq 3$, observe $\#S(p) \leq 2$ since $x^2 + x + 1$ is quadratic. Also, since $p$ does not divide the discriminant $D=-3$ of $x^2 + x + 1$, we have by Hensel's Lemma \cite[II.2.2]{lang} that the solutions of $x^2 + x + 1 \equiv 0 \pmod{p}$ lift uniquely to solutions of  $x^2 + x + 1 \equiv 0 \pmod{p^r}$. This yields  $\#S(p^r) = \#S(p) \leq 2$ for all $r \geq 1$.
    Then applying the Chinese Remainder Theorem to the fact that $\#S(p^r) \leq 2$ for all prime powers $p^r$, we obtain $\#S(N) \leq 2^{\omega(N)}$, as desired.

    An identical argument (with special case of $p=2$) shows that $\#S'(N) \leq 2^{\omega(N)}$ as well.
\end{proof}

Next, we state Lemma \ref{lem:phi-sum}, which will be used to bound the fourth term from Formula \ref{form:fullspace-dim}. We will also prove a more general version of this lemma later in Formula \ref{form:beta-conv-sigma}.
\begin{lemma}[{\cite[Definition~12\textquotesingle(B),~Equation~(7)]{martin}}] \label{lem:phi-sum}
    Let $N \geq 1$. Then 
    \begin{align}
        \sum_{d \mid N} \phi(\gcd(d,N/d)) \leq 2^{\omega(N)} \sqrt{N}.
    \end{align}
\end{lemma}

Finally, we prove Lemma \ref{lem:omega-sigma0-pi-bounds} to bound the multiplicative functions  
$2^{\omega(N)}$ and $\nu(N)$. Both of these functions are $O(N^{\varepsilon})$ for any $\varepsilon > 0$ \cite[Sections~18.1,~22.13]{hardy-wright}. But Lemma \ref{lem:omega-sigma0-pi-bounds} gives more explicit bounds, which we will need later.
\begin{lemma} \label{lem:omega-sigma0-pi-bounds}
    Let $\omega(N)$ denote the number of distinct prime divisors of $N$, and let \\ 
    $ \displaystyle
        \nu(N) := \prod_{p\mid N} \begin{cases} 
            4 & \text{if } p=2 \\
            \lrp{1+\frac{2}{p-2}}  & \text{if } p \neq 2
        \end{cases}
    $. Then $
        2^{\omega(N)} \leq 4.862 \cdot N^{1/4} 
        $
        and 
        $
        \nu(N) \leq 21.234 \cdot N^{1/16}
        $.
\end{lemma}
\begin{proof}
    For the $2^{\omega(N)}$ bound, observe that $2 \leq p^{1/4}$ for $p \geq 17$. So let $c_p = 1$ for $p \geq 17$, and $c_p = 2/\,p^{1/4}$ for $2 \leq p \leq 13$. Then
    \begin{align}
        2^{\omega(N)} &= \prod_{p \mid N} 2 \\
        &\leq \prod_{p \mid N} c_p p^{1/4} \\
        &\leq \prod_{p^r \midmid N} c_p p^{r/4} \\
        &\leq c_2 \cdots c_{13} \cdot N^{1/4} \\
        &\leq 4.862 \cdot N^{1/4}.
    \end{align}
    
    Similarly, for the $\nu(N)$ bound, one can easily verify that $1 + \frac{2}{p-2} \leq p^{1/16}$ for primes $p \geq 17$. So let $c_p' = 1$ for $p \geq 17$, 
    $c_2' = 4 /\,2^{1/16}$, and $c_p' = \lrp{1 + \frac{2}{p-2}}/\,p^{1/16}$ for $3 \leq p \leq 13$. Then by an identical argument,
    \begin{align}
        \nu(N) 
        &= \prod_{p \mid N} \begin{cases} 
            4 & \text{if } p=2 \\
            \lrp{1+\frac{2}{p-2}}  & \text{if } p \neq 2
        \end{cases} \\
        &\leq c_2' \cdots c_{13}' \cdot N^{1/16} \\
        &\leq 21.234 \cdot N^{1/16},
    \end{align}
    completing the proof.
\end{proof}

We now have the tools to prove Theorem \ref{thm:fullspace-dim}.

{
\renewcommand{\thetheorem}{\ref{thm:fullspace-dim}}
\begin{theorem}
    Consider $N \geq 1$, $k \geq 2$, and $\chi$ a Dirichlet character modulo $N$ such that $\chi(-1) = (-1)^k$. Then for all bounds $B \geq 0$, $\dim S_k(\Gamma_0(N), \chi) \leq B$ for only finitely many triples $(N,k,\chi)$.
\end{theorem}
\addtocounter{theorem}{-1}
}

\begin{proof}
    We give bounds on each of the terms in Formula \ref{form:fullspace-dim}. 
    
    First, casework modulo $3$ on \eqref{eqn:c3} shows that $\lrabs{c_3(k)} \leq \frac13$. Similarly, casework modulo $4$ on \eqref{eqn:c4} shows that $\lrabs{c_4(k)} \leq \frac12$. 
    
    Next, Lemma \ref{lem:congruence-solns} shows that 
    \begin{align} \label{chi-sum-1-bound}
        \lrabs{\sum_{
            \substack{x \mod N \\  x^2+x+1\equiv0 \mod N}
        }\mspace{-35mu} \chi(x) } 
         &\leq S(N)
         \leq  2^{\omega(N)}
    \end{align}
    and
    \begin{align} \label{chi-sum-2-bound}
        \lrabs{\sum_{
            \substack{x \mod N \\  x^2+1\equiv0 \mod N}
        }\mspace{-25mu} \chi(x) }
        &\leq S'(N)
         \leq  2^{\omega(N)}.
    \end{align}

    Finally, Lemma \ref{lem:phi-sum} shows that
    \begin{align}
        \lrabs{\sum_{
            \substack{d \mid N \\  (d,N/d) \mid N / f}
        }\mspace{-20mu} \phi(\gcd(d,N/d)) }
        \ \ &\leq \ \sum_{
            \substack{d \mid N }
        } \phi(\gcd(d,N/d)) 
        \ \leq \ 2^{\omega(N)} \sqrt{N}.
    \end{align}

    Let $B \geq 0$ be a fixed natural number. Then combining each of the above bounds, we obtain
    \begin{align}
        \dim S_k(\Gamma_0(N), \chi) - B \ =& \   
        \frac{k-1}{12} \psi(N) 
        \ - \ c_3(k) \mspace{-30mu}\sum_{
            \substack{x \mod N \\  x^2+x+1\equiv0 \mod N}
        }\mspace{-35mu} \chi(x) 
        \ - \ c_4(k) \mspace{-20mu}\sum_{
            \substack{x \mod N \\  x^2+1\equiv0 \mod N}
        }\mspace{-25mu} \chi(x) \\
        &- \ \frac12 \mspace{-20mu}\sum_{
            \substack{d \mid N \\  (d,N/d) \mid N / f}
        }\mspace{-20mu} \phi(\gcd(d,N/d))
        \ + \ c_0(k,\chi) \ - \ B \\
        \geq& \ \frac{k-1}{12} \psi(N) - \frac13 \cdot 2^{\omega(N)} - \frac12 \cdot 2^{\omega(N)} - \frac12 \cdot 2^{\omega(N)} \sqrt{N} - B \\
        =& \, \psi(N) \lrb{
          \frac{k-1}{12} - \frac{ \frac{5}{6} \cdot 2^{\omega(N)} + \frac12 \cdot  2^{\omega(N)} \sqrt{N} + B}{\psi(N)} 
        } \\
        =& \, \psi(N) \lrb{
          \frac{k-1}{12} - E(N) 
        }, \label{eqn:fullspace-dim-bound}
    \end{align}
    where
    $$E(N) = \frac{ \frac{5}{6} \cdot 2^{\omega(N)} + \frac12 \cdot  2^{\omega(N)} \sqrt{N} + B}{\psi(N)}.$$
    Now, $2^{\omega(N)} = O(N^{\varepsilon})$ 
    for any $\varepsilon > 0$, and $\psi(N) \geq N$ from \eqref{eqn:psi}. 
    Thus
    \begin{align}
         E(N) = \frac{O(N^{1/2 + \varepsilon})}{\psi(N)} = O(N^{-1/2+\varepsilon}).
    \end{align}
    This means that for $N$ sufficiently large, we will have $E(N) < \frac{1}{12}$, and hence by
    \eqref{eqn:fullspace-dim-bound}, 
    \begin{align}
        \dim S_k(\Gamma_0(N), \chi) - B > \psi(N) \lrb{\frac{k-1}{12} - \frac{1}{12}} \geq 0.
    \end{align}
    So for sufficiently large $N$, we have $\dim S_k(\Gamma_0(N), \chi) > B$ (independently of $k$ and $\chi$). 
    Then for each of the finitely many remaining values of $N$, $\frac{k-1}{12}$ will be $ > E(N)$ for sufficiently large $k$. 
    This means that $\dim S_k(\Gamma_0(N), \chi) \leq B$ for only finitely many triples $(N,k,\chi)$.
\end{proof}

We note here that since $S_k(\Gamma_0(N), \chi)$ is contained in $M_k(\Gamma_0(N), \chi)$ and $S_k(\Gamma_0(N))$ is contained in $M_k(\Gamma_1(N))$, $S_k(\Gamma_1(N))$, and $M_k(\Gamma_1(N))$, we also have the same result for those spaces.

\begin{corollary}
    Consider $N \geq 1$, $k \geq 2$, and $\chi$ a Dirichlet character modulo $N$ such that $\chi(-1) = (-1)^k$. Then for all bounds $B \geq 0$, $\dim M_k(\Gamma_0(N), \chi) \leq B$ for only finitely many triples $(N,k,\chi)$. Similarly, for $N \geq 1$ and $k \geq 2$ even, 
    $\dim S_k(\Gamma_0(N))$, $\dim M_k(\Gamma_0(N))$, $\dim S_k(\Gamma_1(N))$, and $\dim M_k(\Gamma_1(N))$ are $\leq B$ for only finitely many pairs $(N,k)$.
\end{corollary}

By computing the bounds given in Theorem \ref{thm:fullspace-dim}, we also give the complete list of triples $(N,k,\chi)$ for which $\dim S_k(\Gamma_0(N), \chi) \leq 2$. To find $N$ large enough such that $E(N) < \frac{1}{12}$, we use the explicit bound  $2^{\omega(N)} \leq 4.862 \cdot N^{1/4}$ from Lemma \ref{lem:omega-sigma0-pi-bounds}.
This yields
\begin{align}
    E(N) &= \frac{\frac56 \cdot 2^{\omega(N)} + \frac12 \cdot 2^{\omega(N)} \sqrt{N} + 2}{\psi(N)} \\
    &\leq  \frac{ \frac56 \cdot 4.862 \cdot N^{1/4} + \frac12 \cdot 4.862 \cdot N^{3/4} + 2}{N},
\end{align}
which will be $< \frac{1}{12}$ for all $N \geq 729974$.
Then we just need to check all the $N < 729974$ by computer; see \cite{ross-code} for the code. This yields the complete list given in Tables \ref{table:fullspace-dim0}, \ref{table:fullspace-dim1}, and \ref{table:fullspace-dim2}.

In these tables, we identify the characters $\chi$ by their Conrey label \cite{lmfdb-conrey-label}. If one is only interested in trivial character, these tables of course also determine the complete list of pairs $(N,k)$ for which $\dim S_k(\Gamma_0(N)) \leq 2$. These would be the entries with Conrey label $1$.

\newpage 
 
\begin{mytable} \label{table:fullspace-dim0}
\begin{equation}
\setlength{\arraycolsep}{1mm}
\begin{array}{|c|c|c|c|c|c|c|c|c|}
\hline
\multicolumn{9}{|c|}{\makecell{
    \text{All triples $(N,k,\chi)$ for which $\dim S_k(\Gamma_0(N),\chi) = 0$.} \\
    \text{The characters $\chi$ here are identified by their Conrey label.}
  }} \\
\hline
(1, 2, 1) & (1, 4, 1) & (1, 6, 1) & (1, 8, 1) & (1, 10, 1) & (1, 14, 1) & (2, 2, 1) & (2, 4, 1) & (2, 6, 1) \\
\hline
(3, 2, 1) & (3, 3, 2) & (3, 4, 1) & (3, 5, 2) & (4, 2, 1) & (4, 3, 3) & (4, 4, 1) & (5, 2, 1) & (5, 2, 4) \\
\hline
(5, 3, 2) & (5, 3, 3) & (5, 4, 4) & (6, 2, 1) & (6, 3, 5) & (7, 2, 1) & (7, 2, 2) & (7, 2, 4) & (7, 3, 3) \\
\hline
(7, 3, 5) & (8, 2, 1) & (8, 2, 5) & (8, 3, 7) & (9, 2, 1) & (9, 2, 4) & (9, 2, 7) & (9, 3, 8) & (10, 2, 1) \\
\hline
(10, 2, 9) & (11, 2, 3) & (11, 2, 4) & (11, 2, 5) & (11, 2, 9) & (12, 2, 1) & (12, 2, 11) & (13, 2, 1) & (13, 2, 3) \\
\hline
(13, 2, 9) & (13, 2, 12) & (14, 2, 9) & (14, 2, 11) & (15, 2, 2) & (15, 2, 4) & (15, 2, 8) & (16, 2, 1) & (16, 2, 9) \\
\hline
(17, 2, 4) & (17, 2, 13) & (17, 2, 16) & (18, 2, 1) & (19, 2, 7) & (19, 2, 11) & (20, 2, 9) & (21, 2, 20) & (24, 2, 23) \\
\hline
(25, 2, 1) & (25, 2, 24) & (27, 2, 10) & (27, 2, 19) & (32, 2, 17)  & \multicolumn{4}{c|}{~} \\
\hline
\end{array}
\end{equation}
\end{mytable}

\begin{mytable} \label{table:fullspace-dim1}
\begin{equation}
\setlength{\arraycolsep}{1mm}
\begin{array}{|c|c|c|c|c|c|c|c|c|}
\hline
\multicolumn{9}{|c|}{\makecell{
    \text{All triples $(N,k,\chi)$ for which $\dim S_k(\Gamma_0(N),\chi) = 1$.} \\
    \text{The characters $\chi$ here are identified by their Conrey label.}
  }} \\
\hline
(1, 12, 1) & (1, 16, 1) & (1, 18, 1) & (1, 20, 1) & (1, 22, 1) & (1, 26, 1) & (2, 8, 1) & (2, 10, 1) & (3, 6, 1) \\
\hline
(3, 7, 2) & (3, 8, 1) & (4, 5, 3) & (4, 6, 1) & (5, 4, 1) & (5, 5, 2) & (5, 5, 3) & (5, 6, 1) & (6, 4, 1) \\
\hline
(7, 3, 6) & (7, 4, 1) & (7, 4, 2) & (7, 4, 4) & (7, 5, 6) & (8, 3, 3) & (8, 4, 1) & (9, 3, 2) & (9, 3, 5) \\
\hline
(9, 4, 1) & (10, 3, 3) & (10, 3, 7) & (11, 2, 1) & (11, 3, 2) & (11, 3, 6) & (11, 3, 7) & (11, 3, 8) & (11, 3, 10) \\
\hline
(12, 3, 5) & (13, 2, 4) & (13, 2, 10) & (13, 3, 2) & (13, 3, 6) & (13, 3, 7) & (13, 3, 11) & (14, 2, 1) & (15, 2, 1) \\
\hline
(16, 2, 5) & (16, 2, 13) & (16, 3, 15) & (17, 2, 1) & (17, 2, 2) & (17, 2, 8) & (17, 2, 9) & (17, 2, 15) & (18, 2, 7) \\
\hline
(18, 2, 13) & (19, 2, 1) & (19, 2, 4) & (19, 2, 5) & (19, 2, 6) & (19, 2, 9) & (19, 2, 16) & (19, 2, 17) & (20, 2, 1) \\
\hline
(20, 2, 3) & (20, 2, 7) & (21, 2, 1) & (21, 2, 4) & (21, 2, 5) & (21, 2, 16) & (21, 2, 17) & (22, 2, 3) & (22, 2, 5) \\
\hline
(22, 2, 9) & (22, 2, 15) & (23, 2, 2) & (23, 2, 3) & (23, 2, 4) & (23, 2, 6) & (23, 2, 8) & (23, 2, 9) & (23, 2, 12) \\
\hline
(23, 2, 13) & (23, 2, 16) & (23, 2, 18) & (24, 2, 1) & (25, 2, 6) & (25, 2, 11) & (25, 2, 16) & (25, 2, 21) & (26, 2, 3) \\
\hline
(26, 2, 9) & (27, 2, 1) & (28, 2, 9) & (28, 2, 25) & (29, 2, 7) & (29, 2, 16) & (29, 2, 20) & (29, 2, 23) & (29, 2, 24) \\
\hline
(29, 2, 25) & (31, 2, 2) & (31, 2, 4) & (31, 2, 8) & (31, 2, 16) & (32, 2, 1) & (36, 2, 1) & (37, 2, 10) & (37, 2, 26) \\
\hline
(49, 2, 1) & (49, 2, 18) & (49, 2, 30)  & \multicolumn{6}{c|}{~} \\
\hline
\end{array}
\end{equation}
\end{mytable} 

\begin{mytable} \label{table:fullspace-dim2}
\begin{equation}
\setlength{\arraycolsep}{1mm}
\begin{array}{|c|c|c|c|c|c|c|c|c|}
\hline
\multicolumn{9}{|c|}{\makecell{
    \text{All triples $(N,k,\chi)$ for which $\dim S_k(\Gamma_0(N),\chi) = 2$.} \\
    \text{The characters $\chi$ here are identified by their Conrey label.}
  }} \\
\hline
(1, 24, 1) & (1, 28, 1) & (1, 30, 1) & (1, 32, 1) & (1, 34, 1) & (1, 38, 1) & (2, 12, 1) & (2, 14, 1) & (3, 9, 2) \\
\hline
(3, 10, 1) & (3, 11, 2) & (4, 7, 3) & (4, 8, 1) & (5, 6, 4) & (5, 7, 2) & (5, 7, 3) & (5, 8, 4) & (6, 5, 5) \\
\hline
(7, 5, 3) & (7, 5, 5) & (7, 6, 2) & (7, 6, 4) & (8, 4, 5) & (8, 5, 7) & (9, 4, 4) & (9, 4, 7) & (9, 5, 8) \\
\hline
(10, 4, 9) & (11, 4, 1) & (11, 4, 3) & (11, 4, 4) & (11, 4, 5) & (11, 4, 9) & (12, 3, 7) & (13, 3, 5) & (13, 3, 8) \\
\hline
(13, 4, 4) & (13, 4, 10) & (13, 4, 12) & (14, 3, 3) & (14, 3, 5) & (14, 3, 13) & (15, 3, 7) & (15, 3, 11) & (15, 3, 13) \\
\hline
(15, 3, 14) & (16, 3, 7) & (17, 3, 3) & (17, 3, 5) & (17, 3, 6) & (17, 3, 7) & (17, 3, 10) & (17, 3, 11) & (17, 3, 12) \\
\hline
(17, 3, 14) & (18, 3, 17) & (19, 3, 2) & (19, 3, 3) & (19, 3, 10) & (19, 3, 13) & (19, 3, 14) & (19, 3, 15) & (22, 2, 1) \\
\hline
(23, 2, 1) & (24, 2, 11) & (24, 2, 13) & (25, 2, 4) & (25, 2, 9) & (25, 2, 14) & (25, 2, 19) & (25, 3, 7) & (25, 3, 18) \\
\hline
(26, 2, 1) & (26, 2, 17) & (26, 2, 23) & (26, 2, 25) & (27, 2, 4) & (27, 2, 7) & (27, 2, 13) & (27, 2, 16) & (27, 2, 22) \\
\hline
(27, 2, 25) & (28, 2, 1) & (28, 2, 3) & (28, 2, 19) & (28, 2, 27) & (29, 2, 1) & (29, 2, 4) & (29, 2, 5) & (29, 2, 6) \\
\hline
(29, 2, 9) & (29, 2, 13) & (29, 2, 22) & (29, 2, 28) & (30, 2, 17) & (30, 2, 19) & (30, 2, 23) & (31, 2, 1) & (31, 2, 5) \\
\hline
(31, 2, 7) & (31, 2, 9) & (31, 2, 10) & (31, 2, 14) & (31, 2, 18) & (31, 2, 19) & (31, 2, 20) & (31, 2, 25) & (31, 2, 28) \\
\hline
(32, 2, 9) & (32, 2, 25) & (33, 2, 2) & (33, 2, 4) & (33, 2, 8) & (33, 2, 16) & (33, 2, 17) & (33, 2, 25) & (33, 2, 29) \\
\hline
(33, 2, 31) & (33, 2, 32) & (34, 2, 13) & (34, 2, 21) & (34, 2, 33) & (35, 2, 3) & (35, 2, 4) & (35, 2, 9) & (35, 2, 11) \\
\hline
(35, 2, 12) & (35, 2, 13) & (35, 2, 16) & (35, 2, 17) & (35, 2, 27) & (35, 2, 29) & (35, 2, 33) & (36, 2, 35) & (37, 2, 1) \\
\hline
(37, 2, 7) & (37, 2, 9) & (37, 2, 11) & (37, 2, 12) & (37, 2, 16) & (37, 2, 27) & (37, 2, 33) & (37, 2, 34) & (37, 2, 36) \\
\hline
(39, 2, 5) & (39, 2, 8) & (39, 2, 25) & (40, 2, 7) & (40, 2, 9) & (40, 2, 23) & (41, 2, 4) & (41, 2, 10) & (41, 2, 16) \\
\hline
(41, 2, 18) & (41, 2, 23) & (41, 2, 25) & (41, 2, 31) & (41, 2, 37) & (41, 2, 40) & (43, 2, 4) & (43, 2, 11) & (43, 2, 16) \\
\hline
(43, 2, 21) & (43, 2, 35) & (43, 2, 41) & (45, 2, 8) & (45, 2, 17) & (45, 2, 19) & (48, 2, 47) & (50, 2, 1) & (50, 2, 49) \\
\hline
(64, 2, 33)  & \multicolumn{8}{c|}{~} \\
\hline
\end{array}
\end{equation}
\end{mytable}



\section{The dimension of \texorpdfstring{$S_k^\nw(\Gamma_0(N),\chi)$}{S\_k\^new(Gamma0(N),chi)}  }    \label{sec:dim-formula-newspace}
Next, we turn our attention to the newspace. We start by giving the classical dimension formula for $S_k^\nw(\Gamma_0(N), \chi)$. 
\begin{formula}[{\cite[Corollary~13.3.7]{cohen-stromberg}}]  \label{form:newspace-dim}
    Let $\beta$ be the multiplicative function defined by
    \begin{equation} \label{eqn:beta-def}
        \beta(p^r) = \begin{cases}
            -2 & \text{if } r=1 \\
            1 & \text{if } r=2 \\
            0 & \text{if } r \geq 3,
        \end{cases} 
    \end{equation}
    and let $N \geq 1$, $k \geq 2$, and $\chi$ be a Dirichlet character modulo $N$ of conductor $f=\condf(\chi)$ such that $\chi(-1) = (-1)^k$. 
    Then 
    \begin{align}
        \dim S_k^\nw(\Gamma_0(N),\chi) = \sum_{f \mid M \mid N} \beta(N/M) \dim S_k(\Gamma_0(M), \chi).
    \end{align}
\end{formula}

We then write each of the $\dim S_k(\Gamma_0(N), \chi)$ terms appearing in Formula \ref{form:newspace-dim} as a linear combination of multiplicative functions. In the following, $v_p(\cdot)$ denotes $p$-adic valuation.

\begin{lemma} \label{lem:dim-linear-comb-mult}
    Let $N \geq 1$, $k\geq 2$, and $\chi$ be a primitive character modulo $f$ with $f \mid N$ and where $\chi(-1) = (-1)^k$. Factor $N = p_1^{r_1} \cdots p_t^{r_t}$, and write $f = p_1^{\alpha_1} \cdots p_t^{\alpha_t}$ (the $\alpha_i$ are possibly $0$ here). Then factor $\chi$ as $\chi = \chi_{p_1^{\alpha_1}} \cdots \chi_{p_t^{\alpha_t}}$ where each $\chi_{p^\alpha}$ is a primitive character modulo $p^\alpha$ \cite[p.~71]{stein}. We also consider $\chi_{p^0}$ to be the trivial character modulo $1$ even for $p \nmid N$. Then define the three arithmetic functions $\rho$, $\rho'$, and $\sigma$ on prime powers as 
    \begin{align}
        \rho(p^r) &:= 
        \begin{dcases}
            \sum_{
                \substack{x \mod p^r \\  x^2+x+1\equiv0 \mod p^r}
            }\mspace{-35mu} \chi_{p^\alpha}(x) &\text{ if } r \geq \alpha  \ \ (\text{where } \alpha=v_p(f)) \\
            1 &\text{ if } 1 \leq r < \alpha
        \end{dcases}   \label{eqn:rho-def}  \\
        \rho'(p^r) &:= 
        \begin{dcases}
            \sum_{
                \substack{x \mod p^r \\  x^2+1\equiv0 \mod p^r}
            }\mspace{-35mu} \chi_{p^\alpha}(x) &\text{ if } r \geq \alpha  \ \ (\text{where } \alpha=v_p(f)) \\
            1 &\text{ if } 1 \leq r < \alpha
        \end{dcases}   \label{eqn:rho-prime-def}  \\
        \sigma(p^r) &:= 
        \begin{dcases}
            \sum_{
                \substack{d \mid p^r \\   \gcd(d,p^r/d) \mid p^{r-\alpha} }
            }\mspace{-20mu} \phi(\gcd(d, p^r / d))
            &\text{ if } r \geq \alpha  \ \ (\text{where } \alpha=v_p(f)) \\
            1 &\text{ if } 1 \leq r < \alpha,
        \end{dcases} \label{eqn:sigma-def} 
    \end{align}
    and extend these three functions multiplicatively. 
    Then 
    \begin{align}
        \dim S_k(\Gamma_0(N),\chi) = \frac{k-1}{12} \psi(N) - c_3(k) \rho(N) - c_4(k) \rho'(N) - \frac{1}{2} \sigma(N) + c_0(k,\chi).
    \end{align}
\end{lemma}
\begin{proof}
    Recall that the space $S_k(\Gamma_0(N),\chi)$ is unchanged whether we consider $\chi$ to be a character modulo $N$ or modulo $f$. And none of the terms in Formula \ref{form:fullspace-dim} are affected in either case. So it is perfectly well-defined to ask about $S_k(\Gamma_0(N),\chi)$ here and use Formula \ref{form:fullspace-dim} to compute its dimension. 
    
    Now, comparing \eqref{eqn:rho-def} with Formula \ref{form:fullspace-dim}, we first need to show that 
    \begin{align}
        \rho(N) = \mspace{-30mu}\sum_{
            \substack{x \mod N \\  x^2+x+1\equiv0 \mod N}
        }\mspace{-35mu} \chi(x).
    \end{align}
    Recall that $N = p_1^{r_1} \cdots p_t^{r_t}$, and observe that $r_i \geq \alpha_i$ for $1 \leq i \leq t$. Thus by \eqref{eqn:rho-def} and the Chinese Remainder Theorem,
    \begin{align}
        \rho(N) 
        &= \prod_{1 \leq i \leq t} \rho(p_i^{r_i}) \\
        &= \prod_{1 \leq i \leq t}  \sum_{
                \substack{x_i \mod p_i^{r_i} \\  x_i^2+x_i+1\equiv0 \mod p_i^{r_i}}
            }\mspace{-35mu} \chi_{p_i^{\alpha_i}}(x_i) \\
        &= \sum_{\substack{
            x_1^2 + x_1 + 1 \equiv 0 \mod p_1^{r_1} \\ 
               \vspace{-3.3mm} \\ 
               \vdots \\
            x_t^2 + x_t + 1 \equiv 0 \mod p_t^{r_t}
        }} \ \prod_{1 \leq i \leq t} \chi_{p_i^{\alpha_i}}(x_i) \\
        &= \sum_{ \substack{x \mod N \\  x^2+x+1\equiv0 \mod N} } \ \prod_{1 \leq i \leq t} \chi_{p_i^{\alpha_i}}(x) \\
        &= \sum_{ \substack{x \mod N \\  x^2+x+1\equiv0 \mod N} } \chi(x).
    \end{align}
    
    By an identical argument, we also have that
    \begin{equation}
        \rho'(N) = \mspace{-30mu}\sum_{
            \substack{x \mod N \\  x^2+1\equiv0 \mod N}
        }\mspace{-35mu} \chi(x) \,.
    \end{equation}

    Then comparing \eqref{eqn:sigma-def} with Formula \ref{form:fullspace-dim}, we also need to show that 
    \begin{equation}
        \sigma(N) = \sum_{
            \substack{d \mid N \\  \gcd(d,N/d) \mid N / f}
        }\mspace{-20mu} \phi(\gcd(d,N/d)).
    \end{equation}
    
    By \eqref{eqn:sigma-def} and the Chinese Remainder Theorem,
    \begin{align}
        \sigma(N) &= \prod_{1 \leq i \leq t} \sigma(p_i^{r_i}) \\
        &= \prod_{1 \leq i \leq t} \sum_{
            \substack{d_i \mid p_i^{r_i} \\  \gcd(d_i,p_i^{r_i}/d_i) \mid {p_i}^{r_i-\alpha_i} }
        }\mspace{-20mu} \phi(\gcd(d_i, p_i^{r_i} / d_i)) \\
        &= \sum_{
        \substack{
            d_1 \mid p_1^{r_1} \\ 
            \gcd(d_1,p_1^{r_1}/d_1) \mid {p_1}^{r_1-\alpha_1} \\
            \vspace{-3.3mm} \\
            \vdots \\
            d_t \mid p_t^{r_t} \\ 
            \gcd(d_t,p_t^{r_t}/d_t) \mid {p_t}^{r_t-\alpha_t} \\
        }
        }\mspace{-10mu} 
        \prod_{1 \leq i \leq t}
        \phi(\gcd(d_i, p_i^{r_i} / d_i)) \\
        &= \sum_{
        \substack{
            d_1 \mid p_1^{r_1} \\ 
            \gcd(d_1,p_1^{r_1}/d_1) \mid {p_1}^{r_1-\alpha_1} \\
            \vspace{-3.3mm} \\
            \vdots \\
            d_t \mid p_t^{r_t} \\ 
            \gcd(d_t,p_t^{r_t}/d_t) \mid {p_t}^{r_t-\alpha_t} \\
        }
        }\mspace{-10mu} 
        \phi\lrp{ \prod_{1 \leq i \leq t} \gcd(d_i, p_i^{r_i} / d_i) } \\
        &= \sum_{\substack{ d \mid N \\ \gcd(d,N/d) \mid N/f}} \phi(\gcd(d, N/d)).
    \end{align}

    Then substituting these three identities into Formula \ref{form:fullspace-dim} yields
    \begin{align}
        \dim S_k(\Gamma_0(N),\chi) = \frac{k-1}{12} \psi(N) - c_3(k) \rho(N) - c_4(k) \rho'(N) - \frac{1}{2} \sigma(N) + c_0(k,\chi),
    \end{align}
    as desired. 
\end{proof}

\section{A theory of partial Dirichlet convolutions} \label{sec:partial-convolutions}

Now that we have written $\dim S_k(\Gamma_0(N),\chi)$ as a linear combination of multiplicative functions, observe that the summation in Formula \ref{form:newspace-dim} resembles a Dirichlet convolution with some of the terms missing. We will refer to sums of this form as ``partial Dirichlet convolutions". We then give a lemma to compute these partial Dirichlet convolutions.
\begin{lemma} \label{lem:partial-convolution} 
    Let $f \geq 1$ be a fixed positive integer. Then for multiplicative functions $\beta$ and $\theta$ with $\theta(f) \neq 0$, define the function $\theta_f$ by $\theta_f(N): = \theta(fN) / \theta(f)$. Then $\theta_f$ is also a multiplicative function, where
    \begin{align} \label{eqn:theta-sub-f}
        \theta_f(p^r) = \begin{dcases}
            \theta(p^r) & \text{if } v_p(f) = 0 \\
            \frac{\theta(p^{r+v_p(f)})}{\theta(p^{v_p(f)})} & \text{in general}. \\
        \end{dcases}
    \end{align}
    Furthermore, the partial convolution of $\beta$ with $\theta$ is given by
    \begin{align}
        \sum_{f \mid M \mid N} \beta(N/M) \theta(M) = \theta(f) \cdot \beta {*} \theta_f(N/f).
    \end{align}
     Here, $v_p(\cdot)$ denotes $p$-adic valuation, and $*$ denotes ordinary Dirichlet convolution.
\end{lemma}
\begin{proof}
    First, we show that $\theta_f(N)$ actually is a multiplicative function of the claimed form. We have,
    \begin{align}
        \theta_f(N) &= \frac{\theta(fN)}{\theta(f)} \\
        &= \prod_{p \mid fN} \frac{\theta(p^{v_p(fN)})}{\theta(p^{v_p(f)})} \\
        &= \prod_{p \mid fN}
        \begin{dcases}
            \frac{\theta(p^{v_p(f)+v_p(N)})}{\theta(p^{v_p(f)})} & \text{ if } v_p(N) \geq 1 \\
            \frac{\theta(p^{v_p(f)})}{\theta(p^{v_p(f)})} & \text{ if } v_p(N) = 0
        \end{dcases} \\
        &= \prod_{p \mid N} \frac{\theta(p^{v_p(f)+v_p(N)})}{\theta(p^{v_p(f)})} \\
        &= \prod_{p \mid N} \begin{dcases}
            \theta(p^{v_p(N)}) & \text{if } v_p(f) = 0 \\
            \frac{\theta(p^{v_p(N)+v_p(f)})}{\theta(p^{v_p(f)})} & \text{in general},
        \end{dcases}
    \end{align}
    as claimed.

    Next, observe that we can convert the desired summation over $f \mid M \mid N$ into a summation over $L \mid N/f$ where $M = f \cdot L$. This yields
    \begin{align}
        \sum_{f \mid M \mid N} \beta(N/M) \theta(M) 
        &= \sum_{L \mid N/f} \beta(N/(f L)) \theta(fL) \\
        &= \theta(f) \cdot \sum_{L \mid N/f} \beta((N/f)/ L) \theta_f(L) \\
        &= \theta(f) \cdot \beta{*}\theta_f(N/f),
    \end{align}
    as desired.
\end{proof}

We note here that the condition $\theta(f) \neq 0$ in the above lemma is not restrictive. If $\theta(M) = 0$ for all $f \mid M \mid N$, then  $\sum_{f \mid M \mid N} \beta(N/M) \theta(M) = 0$. Otherwise, we can always rewrite this partial 
convolution as $\sum_{f' \mid M \mid N} \beta(N/M) \theta(M)$ where $\theta(f') \neq 0$. Factor $N = p_1^{r_1} \cdots p_t^{r_t}$, and write $f = p_1^{\alpha_1} \cdots p_t^{\alpha_t}$ (the $\alpha_i$ are possibly $0$ here). Then for $1 \leq i \leq t$, let $\delta_i$ be the minimal integer in the range $\alpha_i \leq \delta_i \leq r_i$ such that $\theta(p_i^{\delta_i}) \neq 0$. (These $\delta_i$ are guaranteed to exist, since otherwise, we would have $\theta(M) = 0$ for all $f \mid M \mid N$.) Then let $f' := p_1^{\delta_1} \cdots p_t^{\delta_t}$. Observe that $f'$ is constructed here to be the minimum element (with respect to the divisibility ordering) such that $\theta(f') \neq 0$; if $f' \nmid M$ for some $f \mid M \mid N$, then $\theta(M) = 0$. Thus we can rewrite the partial convolution of $\beta$ with $\theta$ as
\begin{align}
    \sum_{f \mid M \mid N} \beta(N/M) \theta(M) = \sum_{f' \mid M \mid N} \beta(N/M) \theta(M),
\end{align}
where $\theta(f') \neq 0$.

For the particular function $\beta$ given in Formula \ref{form:newspace-dim}, we also note how to compute an ordinary Dirichlet convolution of $\beta$ with $\theta$. This lemma follows immediately from \eqref{eqn:beta-def}.
\begin{lemma} \label{lem:beta-convolution}
    Let $\beta$ be the multiplicative function defined in Formula \ref{form:newspace-dim}, and $\theta$ be any other multiplicative function. Then $\beta {*} \theta$ is the multiplicative function defined by
    \begin{align}
        \beta {*} \theta(p^r) = 
        \begin{dcases}
            \theta(p)-2 & \text{ if } r = 1 \\
            \theta(p^r) - 2\theta(p^{r-1}) + \theta(p^{r-2}) & \text{ if } r \geq 2. \\
        \end{dcases}
    \end{align}
\end{lemma}

\section{An explicit dimension formula for the newspace} \label{sec:explicit-dim-formula}

Throughout this entire section, let $N \geq 1$, $k \geq 2$, and $\chi$ be a Dirichlet character modulo $N$ of conductor $f=\condf(\chi)$ such that $\chi(-1) = (-1)^k$. Also, let $\beta$ denote the multiplicative function defined in Formula \ref{form:newspace-dim}.

Combining Formula \ref{form:newspace-dim} with Lemma \ref{lem:dim-linear-comb-mult} yields
\begin{align} \label{eqn:newspace-dim-sum-mult}
    \dim S_k^\nw(\Gamma_0(N),\chi) &= \sum_{f \mid M \mid N} \beta(N/M) \Bigg[ \frac{k-1}{12} \psi(M) - c_3(k) \rho(M) \\
    & \mspace{150mu} - c_4(k) \rho'(M) - \frac{1}{2} \sigma(M) + c_0(k,\chi) \Bigg].
\end{align}

Lemma \ref{lem:partial-convolution} expresses each of the partial convolutions appearing in this formula using multiplicative functions of the form $\theta$, $\theta_f$, and $\beta {*} \theta_f$ . So in this section, we will use  Lemma \ref{lem:beta-convolution} to compute each of these multiplicative functions. 

First, we compute $\psi$, $\psi_f$, and $\beta {*} \psi_f$.
\begin{formula} \label{form:beta-convolute-psif}
    Let $\psi_f$ be defined as in Lemma \ref{lem:partial-convolution}. For $p$ prime, let $\alpha := v_p(f)$.
    Then  $\psi$, $\psi_f$, and $\beta {*} \psi_f$ are the multiplicative functions defined for $r \geq 1$ by
    \begin{align} 
        \psi(p^r) &=  \lrp{1+\frac1p} p^r \label{eqn:expdim-psi} \\
        \psi_f(p^r) &= 
        \begin{cases}
             \lrp{1+\frac1p} p^r \qquad\qquad & \text{if } \alpha = 0 \\
            p^r & \text{if } \alpha \geq 1
        \end{cases} \label{eqn:expdim-psi-f} \\
        \beta {*} \psi_f(p^r) &= \begin{cases}
            \lrp{1 - \frac{1}{p}} p  & \text{ if } r=1,\ \alpha=0 \\
            \lrp{1 - \frac{1}{p} - \frac{1}{p^2} } p^2  & \text{ if } r=2,\ \alpha=0 \\
            \lrp{ 1 - \frac{1}{p} - \frac{1}{p^2} + \frac{1}{p^3} }  p^r & \text{ if } r\geq3,\ \alpha=0 \\
            \lrp{1 - \frac{2}{p}} p  & \text{ if } r=1,\ \alpha\geq1 \\
            \lrp{1 - \frac{2}{p} + \frac{1}{p^2}} p^r  & \text{ if } r\geq2,\ \alpha\geq1.
        \end{cases} \label{eqn:expdim-beta-psi-f}
    \end{align}
\end{formula}
\begin{proof}
    We already have the first claim of the formula from \eqref{eqn:psi}. 
    
    For the second and third claims, we have by \eqref{eqn:theta-sub-f} that
    \begin{align}
        \psi_f(p^r) &= \begin{dcases}
            \psi(p^{r}) = \lrp{1+\frac1p} p^r  & \text{if } \alpha=0 \\
            \frac{\psi(p^{r+\alpha})}{ \psi( p^\alpha) } = \frac{\lrp{1+\frac1p} p^{r+\alpha} }{\lrp{1+\frac1p} p^\alpha} = p^r & \text{if } \alpha\geq1,
        \end{dcases}
    \end{align}
    as desired.
    
    For the fourth, fifth, and sixth claims, if $\alpha=0$, then by Lemma \ref{lem:beta-convolution} and \eqref{eqn:expdim-psi-f},
    \begin{align}
        \beta {*} \psi_f(p) &= \psi_f(p) - 2 = p-1 = \lrp{1 - \frac{1}{p}} p , \\
        \beta {*} \psi_f(p^2) &= \psi_f(p^2) - 2 \psi_f(p) + 1 = (p+1)p - 2(p+1) + 1 =  \lrp{1 - \frac{1}{p} - \frac{1}{p^2} } p^2 , \\
        \beta {*} \psi_f(p^r) 
        &= \psi_f(p^r) - 2 \psi_f(p^{r-1}) + \psi_f(p^{r-2}) \\
        &= (p+1)p^{r-1} - 2 (p+1)p^{r-2} + (p+1)p^{r-3} \\
        &= p^r - p^{r-1} - p^{r-2} + p^{r-3} \\
        &= \lrp{ 1 - \frac{1}{p} - \frac{1}{p^2} + \frac{1}{p^3} } p^r \qquad \text{for } r \geq 3.
    \end{align}

    For the seventh and eighth claims, if $\alpha\ge1$, then by Lemma \ref{lem:beta-convolution} and \eqref{eqn:expdim-psi-f},
    \begin{align}
        \beta {*} \psi_f(p) &= \psi_f(p) - 2 = \lrp{1 - \frac{2}{p}} p \\
        \beta {*} \psi_f(p^r)  
        &= \psi_f(p^r) - 2 \psi_f(p^{r-1}) + \psi_f(p^{r-2}) = \lrp{1 - \frac{2}{p} + \frac{1}{p^2}} p^r  \qquad \text{for } r\geq 2,
    \end{align}
    completing the proof.
\end{proof}

Next, we compute $\sigma$, $\sigma_f$, and $\beta{*}\sigma_f$. A proof of this formula is given in the appendix.

\begin{formula} \label{form:beta-conv-sigma}
    For $p$ prime, let $\alpha := v_p(f)$.
    Then assuming $r \geq \alpha$, $\sigma$ is the multiplicative function defined for $r \geq 1$ by
    \begin{align}
        \sigma(p^r) &= \begin{cases}
            2 p^{r-\alpha} \qquad\qquad\ \ \ &\text{ if } \alpha \leq r < 2\alpha \\
            2 p^{(r-1)/2}  &\text{ if } r \geq 2\alpha, \ r \text{ odd} \\
            \lrp{1+\frac1p} p^{r/2} &\text{ if } r \geq 2\alpha, \ r \text{ even}. 
        \end{cases}  \label{eqn:expdim-sigma} 
    \end{align}
    Furthermore, $\sigma_f$ and $\beta {*} \sigma_f$ are the multiplicative functions defined for $r \geq 1$ by
    \begin{align}
        \sigma_f(p^r) &=  \begin{cases}
            2 p^{(r-1)/2}  & \text{if $r$ is odd} \\
            \lrp{1+\frac1p}p^{r/2} \qquad\ \ & \text{if $r$ is even} 
        \end{cases} \, \qquad\qquad\qquad\qquad  \text{for } \alpha=0 \label{eqn:expdim-sigma-f-a0} \\
        \sigma_f(p^r) &= \begin{cases}
            p^r & \text{if } r<\alpha \\
            p^{(r+\alpha-1)/2} & \text{if } r \geq \alpha,\ r+\alpha \text{ odd} \\
            \frac12 \lrp{1+\frac1p} p^{(r+\alpha)/2} &\text{if } r \geq \alpha,\ r+\alpha \text{ even} 
        \end{cases} \qquad\qquad\ \! \text{for } \alpha\geq 1  \label{eqn:expdim-sigma-f-a1} \\
        \beta {*} \sigma_f(p^r) &=
        \begin{cases}
            0 & \text{ if $r$ is odd} \\
            p-2 & \text{ if $r=2$} \\
            \lrp{1 - \frac2p + \frac{1}{p^2}} p^{r/2} & \text{ if $r \geq 4$ even}
        \end{cases} 
        \qquad\qquad\qquad \ \ \ \ 
        \text{for } \alpha = 0  \label{eqn:expdim-beta-sigma-f-a0} \\
        \beta {*} \sigma_f(p) &= 
        \begin{cases}
            \frac12\lrp{p-3} \qquad\qquad\qquad & \text{if } \alpha=1 \\
            p-2 & \text{if } \alpha \geq 2 
        \end{cases} 
        \qquad\qquad\qquad\qquad\qquad\!\! \text{for } \alpha \geq 1 \label{eqn:expdim-beta-sigma-f-a1-p} \\
        \beta {*} \sigma_f(p^r) &=
        \begin{cases}
            0 & \text{if } r\geq\alpha+1,\, r+\alpha \text{ odd} \\ 
            \frac12 \lrp{1-\frac2p+\frac{1}{p^2}} p^{(r+\alpha)/2} & \text{if } r \geq \alpha+2,\, r+\alpha \text{ even} \\
            \frac12 \lrp{1-\frac3p+\frac{2}{p^2}} p^{r} & \text{if } r=\alpha  \\
            \lrp{1-\frac2p+\frac{1}{p^2}} p^r & \text{if } r < \alpha.
        \end{cases}
        \ \  \text{for } \alpha \geq 1,\, r \geq 2   \label{eqn:expdim-beta-sigma-f-a1-pr} 
    \end{align}
\end{formula}

Next, we compute $\rho$, $\rho_f$, and $\beta {*} \rho_f$. A proof of this formula is given in the appendix.

\begin{formula} \label{form:beta-conv-rho}
    For $p$ prime, let $\alpha := v_p(f)$, and let $\chi_{p^\alpha}$ be defined as in Lemma \ref{lem:dim-linear-comb-mult}. If $\lrp{\frac{-3}{p}} = 1$, then let $u$ denote a square root of $-3$ modulo $p^\alpha$. Then assuming $r \geq \alpha$, $\rho$ is the multiplicative function defined for $r \geq 1$ by
    \begin{align}
        \rho(p^r) &= 
        \begin{cases}
            0  & \text{if } p=2 \\
            1  & \text{if } p=3;\ r=1 \\
            0  & \text{if } p=3;\ r\geq2 \\
            0  & \text{if }  p\neq2,3;\  \lrp{\frac{-3}{p}} = -1 \\
            2  & \text{if } p\neq2,3;\  \lrp{\frac{-3}{p}} = 1;\  \chi_{p^\alpha}\lrp{\frac{-1+u}{2}} = 1 \\
            -1  & \text{if } p\neq2,3;\  \lrp{\frac{-3}{p}} = 1;\  \chi_{p^\alpha}\lrp{\frac{-1+u}{2}} = \frac{-1 \pm \sqrt{-3}}{2}. \\
        \end{cases} \label{eqn:expdim-rho}
    \end{align}
    Furthermore, if $\rho(f) \neq 0$, then $\rho_f$ and $\beta {*} \rho_f$ are the multiplicative functions defined for $r \geq 1$ by
    \begin{align}
        \rho_f(p^r) &= \begin{cases}
            1 & \text{ if } p=3;\ r=1;\ \alpha=0 \\
            1 & \text{ if }p\neq3;\ \alpha\geq1 \\
            2 & \text{ if } p\neq2,3;\ \lrp{\frac{-3}{p}}=1;\ \alpha=0 \\
            0 & \text{ otherwise}
        \end{cases}  \label{eqn:expdim-rho-f} \\
        \beta {*} \rho_f(p) &= \begin{cases}
            -1 & \text{ if } p=3;\ \alpha=0 \\
            -1 & \text{ if }  p\neq3;\ \alpha\geq1 \\
            0  & \text{ if }  p\neq2,3;\ \lrp{\frac{-3}{p}}=1;\ \alpha=0 \\
            -2 & \text{ otherwise}
        \end{cases} \label{eqn:expdim-beta-rho-f-p} \\
        \beta {*} \rho_f(p^2) &= \begin{cases}
            -1 & \text{ if } p=3;\ \alpha=0 \\
            0 & \text{ if } p\neq3;\ \alpha\geq1 \\
            -1 & \text{ if } p\neq2,3;\ \lrp{\frac{-3}{p}}=1;\ \alpha=0 \\
            1 & \text{ otherwise}
        \end{cases} \label{eqn:expdim-beta-rho-f-p2} \\
        \beta {*} \rho_f(p^r) &= \begin{cases}
            1 & \text{ if }  p=3;\ r=3;\ \alpha=0 \\
            0 & \text{ if }  r\geq3, \text{ otherwise}. 
        \end{cases} \label{eqn:expdim-beta-rho-f-pr}
    \end{align}
\end{formula}

We note that $\chi_{p^\alpha}$ here does not necessarily need to be computed explicitly. One can instead use the identity  $\chi_{p^\alpha}(x) = \chi(\hat{x})$ where $\hat{x}$ is the unique integer modulo $f$ determined by the congruences $\hat{x} \equiv x \mod p^\alpha$ and $\hat{x} \equiv 1 \mod q^{v_q(f)}$ for primes $q \mid f$, $q \neq p$ \cite[p.~71]{stein}.

Also, observe from \eqref{eqn:expdim-rho} that if $\rho(p^{\alpha}) = 0$, then $\rho(p^r) = 0$ for all $r \geq \alpha$. Thus if $\rho(f) = 0$, then $\rho(M) = 0$ as well for all $f \mid M \mid N$, and so $ 
    \sum_{f \mid M \mid N} \beta(N/M) \rho(M) = 0.
$
This means that even though we have not yet defined $\rho_f$ in the case when $\rho(f) = 0$, we can still use the formula 
$
    \sum_{f \mid M \mid N} \beta(N/M) \rho(M) = \rho(f) \cdot \beta {*} \rho_f(N/f).
$
Here, we just consider $\rho(f) \cdot \beta {*} \rho_f(N/f)$ to be $0$ if $\rho(f) = 0$.
The same observation applies to Formula \ref{form:beta-conv-rhopm} below as well.

Next, we compute $\rho'$, $\rho_f'$, and $\beta {*} \rho_f'$. A proof of this formula is given in the appendix.

\begin{formula} \label{form:beta-conv-rhopm}
    For $p$ prime, let $\alpha := v_p(f)$, and let $\chi_{p^\alpha}$ be defined as in Lemma \ref{lem:dim-linear-comb-mult}. If $\lrp{\frac{-1}{p}} = 1$, then let $u'$ denote a square root of $-1$ modulo $p^\alpha$. Then assuming $r \geq \alpha$, $\rho'$ is the multiplicative function defined for $r \geq 1$ by
    \begin{align}
        \rho'(p^r) &= 
        \begin{cases}
            1  & \text{if } p=2;\ r=1 \\
            0  & \text{if } p=2;\ r\geq2 \\
            0  & \text{if }  p\neq2;\  \lrp{\frac{-1}{p}} = -1 \\
            2 & \text{if } p\neq2;\  \lrp{\frac{-1}{p}} = 1;\ \chi_{p^\alpha}\lrp{u'} = 1 \\
            -2 & \text{if } p\neq2;\  \lrp{\frac{-1}{p}} = 1;\ \chi_{p^\alpha}\lrp{u'} = -1 \\
            0 & \text{if } p\neq2;\  \lrp{\frac{-1}{p}} = 1;\  \chi_{p^\alpha}\lrp{u'} = \pm i.
        \end{cases} \label{eqn:expdim-rhopm}
    \end{align}
    Furthermore, if $\rho'(f) \neq 0$, then $\rho_f'$ and $\beta {*} \rho_f'$ are the multiplicative functions defined for $r \geq 1$ by
    \begin{align}
        \rho_f'(p^r) &= \begin{cases}
            1 & \text{ if } p=2;\ r=1;\ \alpha=0 \\
            1 & \text{ if }p\neq2;\ \alpha\geq1 \\
            2 & \text{ if } p\neq2;\ \lrp{\frac{-1}{p}}=1;\ \alpha=0 \\
            0 & \text{ otherwise}
        \end{cases}  \label{eqn:expdim-rhopm-f} \\
        \beta {*} \rho_f'(p) &= \begin{cases}
            -1 & \text{ if } p=2;\ \alpha=0 \\
            -1 & \text{ if }  p\neq2;\ \alpha\geq1 \\
            0  & \text{ if }  p\neq2;\ \lrp{\frac{-1}{p}}=1;\ \alpha=0 \\
            -2 & \text{ otherwise}
        \end{cases} \label{eqn:expdim-beta-rhopm-f-p} \\
        \beta {*} \rho_f'(p^2) &= \begin{cases}
            -1 & \text{ if } p=2;\ \alpha=0 \\
            0 & \text{ if } p\neq2;\ \alpha\geq1 \\
            -1 & \text{ if } p\neq2;\ \lrp{\frac{-1}{p}}=1;\ \alpha=0 \\
            1 & \text{ otherwise}.
        \end{cases} \label{eqn:expdim-beta-rhopm-f-p2} \\
        \beta {*} \rho_f'(p^r) &= \begin{cases}
            1 & \text{ if }  p=2;\ r=3;\ \alpha=0 \\
            0 & \text{ if }  r\geq3, \text{ otherwise}. 
        \end{cases} \label{eqn:expdim-beta-rhopm-f-pr}
    \end{align}
\end{formula}

Finally, for the constant function $\mathbf{1}$, 
we compute $\beta {*} \mathbf{1}_f$. One can obtain immediately from Lemma \ref{lem:beta-convolution} that
\begin{align} \label{eqn:expdim-beta-one-f}
    \beta {*} \mathbf{1}_f(p^r) &=  \begin{cases}
        -1 & \text{if } r=1 \\
        0  & \text{if } r\geq2.
    \end{cases} 
\end{align}
We then recognize that $\beta {*} \mathbf{1}_f = \mu$, the Mobius function.

We can then apply Lemma \ref{lem:partial-convolution} and Formulas \ref{form:beta-convolute-psif} - \ref{form:beta-conv-rhopm} to \eqref{eqn:newspace-dim-sum-mult} and obtain an explicit dimension formula for $S_k^\nw(\Gamma_0(N),\chi)$.
To simplify this formula slightly, we use the facts that $\sigma(f) = 2^{\omega(f)}$ from \eqref{eqn:expdim-sigma} and $\beta {*} \mathbf{1}_f(N/f) = \mu(N/f)$ from \eqref{eqn:expdim-beta-one-f}.

{
\renewcommand{\thetheorem}{\ref{thm:explicit-dim-formula}}
\begin{theorem} 
    Let $N\geq1$, $k\geq2$, and $\chi$ be a Dirichlet character modulo $N$ of conductor $f=\condf(\chi)$ with $\chi(-1) = (-1)^k$. Then 
    \begin{align} 
        \dim S_k^\nw(\Gamma_0(N),\chi) &= \frac{k-1}{12} \cdot \psi(f) \cdot \beta {*} \psi_f(N/f)   - c_3(k) \cdot \rho(f) \cdot \beta {*} \rho_f(N/f)  \\
        & \mspace{20mu} - c_4(k) \cdot \rho'(f) \cdot \beta {*} \rho_f'(N/f)  - \frac{1}{2} \cdot 2^{\omega(f)} \cdot \beta{*}\sigma_f(N/f) + c_0(k,\chi) \cdot  \mu(N/f).
    \end{align}
    Here,
    \begin{align}
        c_3(k) &= \frac{k-1}{3} - \lrfloor{\frac{k}{3}}, \\
        c_4(k) &= \frac{k-1}{4} - \lrfloor{\frac{k}{4}}, \\
        c_0(k,\chi) &= 
        \begin{cases} 
            1 & \text{if $k=2$ and $\chi$ is trivial} \\
            0 & \text{otherwise},
        \end{cases} \\
        \omega(f) \ &\text{denotes the number of distinct prime divisors of $f$}, \\
        \mu \ &\text{denotes the Mobius function},
    \end{align}
    and $\psi$, $\beta{*}\psi_f$, $\rho$, $\beta{*}\rho_f$,  $\rho'$, $\beta{*}\rho_f'$, $\sigma$, and $\beta{*}\sigma_f$ are the multiplicative functions with explicit formulas given in Formulas \ref{form:beta-convolute-psif} - \ref{form:beta-conv-rhopm}.
\end{theorem}
\addtocounter{theorem}{-1}
}

We also implemented and verified this dimension formula against the spaces $S_k^\nw(\Gamma_0(N),\chi)$ computed natively in Sage; see \cite{ross-code} for the code.

To judge size, one can determine the asymptotic behavior in $N$  for each of the terms in this dimension formula. In particular, we have directly from Formulas \ref{form:beta-convolute-psif} - \ref{form:beta-conv-rhopm} that
\begin{equation} \label{eqn:expdim-terms-asymptotics}
\begin{aligned}
    \psi(f) \cdot \beta {*} \psi_f(N/f) &= O(N) \qquad\qquad 
    & \psi(f) \cdot \beta {*} \psi_f(N/f) &= \Omega(N^{1-\varepsilon}) \\
    \rho(f) \cdot \beta{*}\rho_f(N/f) &= O(N^\varepsilon) 
    & \rho'(f) \cdot \beta {*} \rho_f'(N/f) &= O(N^\varepsilon) \\
    2^{\omega(f)} \cdot \beta {*} \sigma_f(N/f) &= O(\sqrt{N})
    & \mu(N/f) &= O(1),
\end{aligned}
\end{equation}
for any $\varepsilon > 0$. These bounds are computed more explicitly in the proof of Theorem \ref{thm:newspace-small}.


\section{When \texorpdfstring{$\dim S_k^\nw(\Gamma_0(N),\chi)$}{dim S\_k\^new(Gamma0(N),chi)} is small}    \label{sec:when-dim-newspace-small}

With this explicit dimension formula, it is now possible to classify small newspaces.
We first show that $S_k^\nw(\Gamma_0(N), \chi)$ is trivial for the infinite family of triples $(N,k,\chi)$ where $2 \mid f$ and $2 \midmid N/f$.

\begin{proposition} \label{prop:infinite-family-trivial-newspace}
    Let $N \geq 1$, $k \geq 2$, and $\chi$ be a Dirichlet character modulo $N$ of conductor $f = \condf(\chi)$ such that $\chi(-1) = (-1)^k$. Then if $2 \mid f$ and $2 \midmid N/f$, we have $\dim S_k^\nw(\Gamma_0(N), \chi) = 0$.
\end{proposition}
\begin{proof}
    First, $2 \mid f$ means that $\chi$ is not trivial, so $c_0(k,\chi) = 0$.
    
    Additionally, observe that $2 \mid f$ implies $4 \mid f$ as well. Because if $2 \midmid f$, then the factorization of $\chi$ (considered as a character modulo $f$) would include a primitive character modulo $2$ \cite[p.~72]{stein}. And this is not possible since there do not exist any primitive characters modulo $2$. 

    Thus for $p=2$, we have $\alpha = v_2(f) \geq 2$.
    So by \eqref{eqn:expdim-beta-psi-f}, \eqref{eqn:expdim-beta-sigma-f-a1-p}, \eqref{eqn:expdim-rho}, and \eqref{eqn:expdim-rhopm},
    \begin{align}
        \beta{*}\psi_f(2) = \beta{*}\sigma_f(2) = 0 \qquad \text{and} \qquad
        \rho(2^\alpha) = \rho'(2^\alpha) = 0.
    \end{align}
    This means that 
    \begin{align}
        \beta{*}\psi_f(N/f) = \beta{*}\sigma_f(N/f) = 0 \qquad \text{and} \qquad
        \rho(f) = \rho'(f) = 0
    \end{align}
    as well,
    since $2 \midmid N/f$ and $2^\alpha \midmid f$. 
    Thus all the terms in the dimension formula from Theorem \ref{thm:explicit-dim-formula} vanish, and so $\dim S_k^\nw(\Gamma_0(N), \chi) = 0$.
\end{proof}

Next, we show that excluding the infinite family given in Proposition \ref{prop:infinite-family-trivial-newspace}, $\dim S_k^\nw(\Gamma_0(N),\chi)$ will be small for only finitely many triples $(N,k,\chi)$.

{
\renewcommand{\thetheorem}{\ref{thm:newspace-small}}
\begin{theorem}
    Consider $N \geq 1$, $k \geq 2$, and $\chi$ a Dirichlet character modulo $N$ of conductor $f=\condf(\chi)$ with $\chi(-1) = (-1)^k$. Also assume it is not the case that $2 \mid f$ and $2 \midmid N/f$. Then for all bounds $B \geq 0$,  $\dim S_k^\nw(\Gamma_0(N),\chi) \leq B$ for only finitely many triples $(N,k,\chi)$.
\end{theorem}
\addtocounter{theorem}{-1}
}

\begin{proof}
    From Theorem \ref{thm:explicit-dim-formula}, we have
    \begin{align} 
        \dim S_k^\nw(\Gamma_0(N),\chi) &= \frac{k-1}{12} \cdot \psi(f) \cdot \beta {*} \psi_f(N/f)   - c_3(k) \cdot \rho(f) \cdot \beta {*} \rho_f(N/f)  \label{eqn:newspace-dim-temp2} \\
        & \ \  - c_4(k) \cdot \rho'(f) \cdot \beta {*} \rho_f'(N/f)  - \frac{1}{2} \cdot 2^{\omega(f)} \cdot \beta{*}\sigma_f(N/f) + c_0(k,\chi) \cdot \mu(N/f).
    \end{align}
    We estimate each of the terms in \eqref{eqn:newspace-dim-temp2}.

    First, we show that $\psi(f) \cdot \beta {*} \psi_f(N/f) \geq N / \nu(N)$, where $\nu$ is the multiplicative function
    \begin{align}
        \nu(N) &:= \prod_{p\mid N} \begin{cases} 
            4 & \text{if } p=2 \\
            \lrp{1+\frac{2}{p-2}}  & \text{if } p \neq 2.
        \end{cases}
    \end{align}
    For $p\neq 2$, in every case of \eqref{eqn:expdim-beta-psi-f}, we have
    \begin{align}
        \beta {*} \psi_f(p^r) &\geq p^r \lrp{1 - \frac{2}{p}} = \frac{p^r}{\lrp{1+\frac{2}{p-2}}}.
    \end{align}
    For $p=2$, observe that the fourth case of \eqref{eqn:expdim-beta-psi-f} occurs in $\beta {*} \psi_f(N/f)$ precisely when $2 \mid f$ and $2 \midmid N/f$. So ignoring this case, we have $\beta {*} \psi_f(2^r) \geq \frac14 2^r$. 
    This gives $\beta {*} \psi_f(p^r) \geq  p^r / \nu(p^r)$ for all $p$, which then yields
    \begin{align}
        \psi(f) \cdot \beta {*} \psi_f(N/f) &\geq \psi(f) \cdot \frac{N/f}{\nu(N/f)} \geq f \cdot \frac{N/f}{\nu(N/f)} \geq \frac{N}{\nu(N)}, \label{eqn:psi-bound-temp}
    \end{align}
    as desired.

    Second, observe from \eqref{eqn:expdim-beta-sigma-f-a0}, \eqref{eqn:expdim-beta-sigma-f-a1-p}, and \eqref{eqn:expdim-beta-sigma-f-a1-pr} that
    \begin{align}
        \lrabs{\beta {*} \sigma_f(p^r)} \leq \begin{cases}
            p^{r/2} & \text{if } \alpha=0 \\
            \frac12 p^{(r+\alpha)/2} & \text{if } \alpha\geq1.
        \end{cases}
    \end{align}
    This then yields
    \begin{align}
        \lrabs{2^{\omega(f)} \cdot \beta{*}\sigma_f(N/f)}
        &\leq 2^{\omega(f)} \cdot \sqrt{N/f} \cdot \prod_{p \mid N/f,\ p \mid f} \frac12  p^{v_p(f)/2} \\
        &= \sqrt{N/f} \cdot \prod_{p \mid N/f,\ p \mid f}  p^{v_p(f)/2} \cdot \prod_{p \mid f,\ p \nmid N/f} 2 \\
        &= \sqrt{N} \cdot \prod_{p \mid f,\ p \nmid N/f} \frac{2}{p^{v_p(f)/2}} \\
        &\leq \sqrt{N} \cdot \frac{2}{\sqrt{2}} \cdot \frac{2}{\sqrt{3}} \\
        &= \frac{2 \sqrt{6}}{3} \sqrt{N}. \label{eqn:sigma-bound-temp}
    \end{align}

    Third, observe from \eqref{eqn:expdim-rho} that $\lrabs{\rho(f)} \leq 2^{\omega(f)}$. Also note from \eqref{eqn:expdim-beta-rho-f-p}, \eqref{eqn:expdim-beta-rho-f-p2}, and \eqref{eqn:expdim-beta-rho-f-pr} that 
    \begin{align}
        \lrabs{\beta {*} \rho_f(p^r)} \leq \begin{cases}
            1 & \text{if } p\neq3,\ \alpha\geq1 \\
            2 & \text{otherwise}. 
        \end{cases}
    \end{align}
    This yields 
    \begin{align}
        \lrabs{\rho(f) \cdot \beta {*} \rho_f(N/f)} 
        &\leq 2^{\omega(f)} \cdot 2 \cdot 2^{\#\{ p \mid N/f \colon p \nmid f\}} \\
        &\leq 2 \cdot 2^{\omega(f)}  2^{\#\{ p \mid N \colon p \nmid f\}} \\
        &= 2 \cdot 2^{\omega(N)}. \label{eqn:rho-bound-temp}
    \end{align}
    An identical argument shows that 
    \begin{align}
        \lrabs{\rho'(f) \cdot \beta {*} \rho_f'(N/f)} 
        &\leq 2 \cdot 2^{\omega(N)} \label{eqn:rhopm-bound-temp}
    \end{align}
    as well.

    Finally, recall that 
    \begin{align} \label{eqn:mobius-bound-temp}
        \lrabs{\mu(N/f)} \leq 1. 
    \end{align}

    Then combining \eqref{eqn:psi-bound-temp}, \eqref{eqn:sigma-bound-temp}, \eqref{eqn:rho-bound-temp}, \eqref{eqn:rhopm-bound-temp}, and \eqref{eqn:mobius-bound-temp}, we obtain
    \begin{align} 
        \dim S_k^\nw(\Gamma_0(N),\chi) - B 
        &= \frac{k-1}{12} \cdot \psi(f) \cdot \beta {*} \psi_f(N/f)   - c_3(k) \cdot \rho(f) \cdot \beta {*} \rho_f(N/f)  \\
        & \ \ \ - c_4(k) \cdot \rho'(f) \cdot \beta {*} \rho_f'(N/f)  - \frac{1}{2} \cdot 2^{\omega(f)} \cdot \beta{*}\sigma_f(N/f) \\
        & \ \ \  + c_0(k,\chi) \cdot \mu(N/f) - B \\
        &\geq \frac{k-1}{12} \cdot \frac{N}{\nu(N)} - \frac13 \cdot 2 \cdot 2^{\omega(N)} \\
        &\qquad - \frac12 \cdot 2\cdot 2^{\omega(N)} - \frac{1}{2} \cdot \frac{2\sqrt{6}}{3} \sqrt{N} - 1 \cdot 1 - B \\
        &= \frac{k-1}{12} \cdot \frac{N}{\nu(N)} - \frac{5}{3} \cdot 2^{\omega(N)} - \frac{\sqrt{6}}{3} \cdot \sqrt{N} - 1 - B \\
        &= \frac{N}{\nu(N)} \lrb{ 
            \frac{k-1}{12} 
            - E(N)
        },  \label{eqn:newspace-dim-bound}
    \end{align}
    where
    \begin{align} \label{eqn:E-N-equation}
        E(N) = \frac53 \cdot \frac{\nu(N) 2^{\omega(N)}}{N} 
            + \frac{\sqrt{6}}{3} \cdot \frac{\nu(N)}{\sqrt{N}}
            + (B+1) \cdot \frac{\nu(N)}{N} .
    \end{align}

    Now, $\nu(N),\  2^{\omega(N)} = O(N^{\varepsilon})$ 
    for any $\varepsilon > 0$, so $
         E(N) = O(N^{-1/2+\varepsilon})
    $.
    This means that for $N$ sufficiently large, we will have $E(N) < \frac{1}{12}$, and hence by
    \eqref{eqn:newspace-dim-bound}, 
    \begin{align}
        \dim S_k^\nw(\Gamma_0(N), \chi) - B > \frac{N}{\nu(N)} \lrb{\frac{k-1}{12} - \frac{1}{12}} \geq 0.
    \end{align}
    Then for each of the finitely many remaining values of $N$, $\frac{k-1}{12}$ will be $ > E(N)$ for sufficiently large $k$. 
    Thus $\dim S_k^\nw(\Gamma_0(N), \chi) \leq B$ for only finitely many triples $(N,k,\chi)$.
\end{proof}

By computing the bounds given in Theorem \ref{thm:newspace-small}, we also give the complete list of triples $(N,k,\chi)$ for which $\dim S_k^\nw(\Gamma_0(N), \chi) \leq 1$. To find $N$ large enough such that $E(N) < \frac{1}{12}$, we use the explicit bounds for $\nu(N)$ and $2^{\omega(N)}$ given in Lemma \ref{lem:omega-sigma0-pi-bounds}.
This yields
\begin{align} \label{eqn:E-N-bounds}
    E(N) &= \frac53 \cdot \frac{\nu(N) 2^{\omega(N)}}{N} 
        + \frac{\sqrt{6}}{3} \cdot \frac{\nu(N)}{\sqrt{N}}
        + (1+1) \cdot \frac{\nu(N)}{N} \\
        &\leq \frac53 \cdot \frac{21.234\cdot N^{1/16} \cdot 4.862\cdot N^{1/4}}{N} 
        + \frac{\sqrt{6}}{3} \cdot \frac{21.234\cdot N^{1/16}}{\sqrt{N}}
        + 2 \cdot \frac{21.234\cdot N^{1/16}}{N},
\end{align}
which will be $< \frac{1}{12}$ for all $N \geq 424094$.
Then we just need to check all the $N < 424094$ by computer; see \cite{ross-code} for the code. This yields the complete list given in Tables
 \ref{table:newspace-dim0} and \ref{table:newspace-dim1}. Along with Proposition \ref{prop:infinite-family-trivial-newspace}, this also completes the proof of Theorem \ref{thm:newspace-trivial}.

\begin{mytable} \label{table:newspace-dim0}
\begin{equation}
\setlength{\arraycolsep}{1mm}
\begin{array}{|c|c|c|c|c|c|c|c|c|}
\hline
\multicolumn{9}{|c|}{\makecell{
    \text{All triples $(N,k,\chi)$ for which $\dim S_k^\nw(\Gamma_0(N),\chi) = 0$, excluding
    the $(N,k,\chi)$ where} \\ 
    \text{$2 \mid f$ and $2 \midmid N/f$. The characters $\chi$ here are identified by their Conrey label.}
  }} \\
\hline
(1, 2, 1) & (1, 4, 1) & (1, 6, 1) & (1, 8, 1) & (1, 10, 1) & (1, 14, 1) & (2, 2, 1) & (2, 4, 1) & (2, 6, 1) \\
\hline
(2, 12, 1) & (3, 2, 1) & (3, 3, 2) & (3, 4, 1) & (3, 5, 2) & (4, 2, 1) & (4, 3, 3) & (4, 4, 1) & (4, 8, 1) \\
\hline
(5, 2, 1) & (5, 2, 4) & (5, 3, 2) & (5, 3, 3) & (5, 4, 4) & (6, 2, 1) & (6, 3, 5) & (7, 2, 1) & (7, 2, 2) \\
\hline
(7, 2, 4) & (7, 3, 3) & (7, 3, 5) & (8, 2, 1) & (8, 2, 5) & (9, 2, 1) & (9, 2, 4) & (9, 2, 7) & (9, 3, 8) \\
\hline
(10, 2, 1) & (10, 2, 9) & (11, 2, 3) & (11, 2, 4) & (11, 2, 5) & (11, 2, 9) & (12, 2, 1) & (12, 2, 11) & (12, 6, 1) \\
\hline
(13, 2, 1) & (13, 2, 3) & (13, 2, 9) & (13, 2, 12) & (14, 2, 9) & (14, 2, 11) & (14, 3, 13) & (15, 2, 2) & (15, 2, 4) \\
\hline
(15, 2, 8) & (16, 2, 1) & (17, 2, 4) & (17, 2, 13) & (17, 2, 16) & (18, 2, 1) & (18, 5, 17) & (19, 2, 7) & (19, 2, 11) \\
\hline
(20, 2, 9) & (21, 2, 20) & (22, 2, 1) & (25, 2, 1) & (25, 2, 24) & (26, 2, 17) & (26, 2, 23) & (27, 2, 10) & (27, 2, 19) \\
\hline
(28, 2, 1) & (32, 2, 17) & (36, 3, 17) & (52, 2, 25) & (60, 2, 1) & (90, 3, 71) & (126, 2, 125)  & \multicolumn{2}{c|}{~} \\
\hline
\end{array}
\end{equation}
\end{mytable}

\begin{mytable} \label{table:newspace-dim1}
\begin{equation}
\setlength{\arraycolsep}{1mm}
\begin{array}{|c|c|c|c|c|c|c|c|c|}
\hline
\multicolumn{9}{|c|}{\makecell{
    \text{All triples $(N,k,\chi)$ for which $\dim S_k^\nw(\Gamma_0(N),\chi) = 1$.} \\
    \text{The characters $\chi$ here are identified by their Conrey label.}
  }} \\
\hline
(1, 12, 1) & (1, 16, 1) & (1, 18, 1) & (1, 20, 1) & (1, 22, 1) & (1, 26, 1) & (2, 8, 1) & (2, 10, 1) & (2, 16, 1) \\
\hline
(2, 18, 1) & (2, 24, 1) & (3, 6, 1) & (3, 7, 2) & (3, 8, 1) & (3, 12, 1) & (4, 5, 3) & (4, 6, 1) & (4, 10, 1) \\
\hline
(4, 12, 1) & (4, 14, 1) & (4, 16, 1) & (4, 20, 1) & (5, 4, 1) & (5, 5, 2) & (5, 5, 3) & (5, 6, 1) & (6, 4, 1) \\
\hline
(6, 6, 1) & (6, 8, 1) & (6, 10, 1) & (6, 14, 1) & (7, 3, 6) & (7, 4, 1) & (7, 4, 2) & (7, 4, 4) & (7, 5, 6) \\
\hline
(8, 3, 3) & (8, 4, 1) & (8, 6, 1) & (9, 3, 2) & (9, 3, 5) & (9, 4, 1) & (9, 6, 1) & (10, 3, 3) & (10, 3, 7) \\
\hline
(10, 4, 1) & (10, 8, 1) & (11, 2, 1) & (11, 3, 2) & (11, 3, 6) & (11, 3, 7) & (11, 3, 8) & (11, 3, 10) & (12, 3, 5) \\
\hline
(12, 4, 1) & (12, 5, 5) & (12, 10, 1) & (13, 2, 4) & (13, 2, 10) & (13, 3, 2) & (13, 3, 6) & (13, 3, 7) & (13, 3, 11) \\
\hline
(14, 2, 1) & (15, 2, 1) & (16, 2, 5) & (16, 2, 13) & (16, 3, 15) & (16, 4, 1) & (17, 2, 1) & (17, 2, 2) & (17, 2, 8) \\
\hline
(17, 2, 9) & (17, 2, 15) & (18, 2, 7) & (18, 2, 13) & (18, 4, 1) & (19, 2, 1) & (19, 2, 4) & (19, 2, 5) & (19, 2, 6) \\
\hline
(19, 2, 9) & (19, 2, 16) & (19, 2, 17) & (20, 2, 1) & (20, 2, 3) & (20, 2, 7) & (20, 3, 13) & (20, 3, 17) & (20, 4, 1) \\
\hline
(20, 6, 1) & (21, 2, 1) & (21, 2, 4) & (21, 2, 5) & (21, 2, 16) & (21, 2, 17) & (22, 2, 3) & (22, 2, 5) & (22, 2, 9) \\
\hline
(22, 2, 15) & (23, 2, 2) & (23, 2, 3) & (23, 2, 4) & (23, 2, 6) & (23, 2, 8) & (23, 2, 9) & (23, 2, 12) & (23, 2, 13) \\
\hline
(23, 2, 16) & (23, 2, 18) & (24, 2, 1) & (24, 4, 1) & (25, 2, 6) & (25, 2, 11) & (25, 2, 16) & (25, 2, 21) & (26, 2, 3) \\
\hline
(26, 2, 9) & (26, 3, 5) & (26, 3, 21) & (27, 2, 1) & (27, 3, 8) & (27, 3, 17) & (28, 2, 9) & (28, 2, 25) & (28, 3, 5) \\
\hline
(28, 3, 17) & (29, 2, 7) & (29, 2, 16) & (29, 2, 20) & (29, 2, 23) & (29, 2, 24) & (29, 2, 25) & (30, 2, 1) & (31, 2, 2) \\
\hline
(31, 2, 4) & (31, 2, 8) & (31, 2, 16) & (32, 2, 1) & (32, 3, 15) & (33, 2, 1) & (34, 2, 1) & (34, 2, 9) & (34, 2, 15) \\
\hline
(34, 2, 19) & (34, 2, 25) & (36, 2, 1) & (36, 2, 13) & (36, 2, 25) & (36, 4, 1) & (37, 2, 10) & (37, 2, 26) & (38, 2, 5) \\
\hline
(38, 2, 9) & (38, 2, 17) & (38, 2, 23) & (38, 2, 25) & (38, 2, 35) & (39, 2, 4) & (39, 2, 10) & (40, 2, 1) & (42, 2, 1) \\
\hline
(44, 2, 1) & (44, 2, 5) & (44, 2, 9) & (44, 2, 25) & (44, 2, 37) & (45, 2, 1) & (46, 2, 1) & (48, 2, 1) & (49, 2, 1) \\
\hline
(49, 2, 18) & (49, 2, 30) & (52, 2, 1) & (52, 2, 17) & (52, 2, 49) & (54, 2, 19) & (54, 2, 37) & (64, 2, 1) & (64, 2, 17) \\
\hline
(64, 2, 49) & (68, 2, 9) & (68, 2, 25) & (68, 2, 49) & (68, 2, 53) & (70, 2, 1) & (72, 2, 1) & (76, 2, 1) & (76, 2, 45) \\
\hline
(76, 2, 49) & (78, 2, 1) & (84, 2, 25) & (84, 2, 37) & (100, 2, 1) & (108, 2, 1) & (108, 2, 37) & (108, 2, 73) & (156, 2, 61) \\
\hline
(156, 2, 133) & (180, 2, 1)  & \multicolumn{7}{c|}{~} \\
\hline
\end{array}
\end{equation}
\end{mytable}


 \section{Disproof of a conjecture from Martin} \label{sec:disproof-conj}

 In \cite[Conjecture 27]{martin}, Greg Martin conjectured that $\dim S_2^\nw(\Gamma_0(N))$ takes on all possible non-negative integers. We show that this is in fact not true. The smallest non-negative integer not achieved by $\dim S_2^\nw(\Gamma_0(N))$ is $67846$. In order to prove this fact, we use the same method and bounds as in \eqref{eqn:newspace-dim-bound} and \eqref{eqn:E-N-bounds}. 

 We take $B = 67846$. Then following \eqref{eqn:newspace-dim-bound} and \eqref{eqn:E-N-bounds}, 
 \begin{align}
    \dim S_2^\nw(\Gamma_0(N)) - 67846 \ge \frac{N}{\nu(N)} \lrb{\frac{k-1}{12} - E(N)}
\end{align}
where 
\begin{align} 
    E(N) &= \frac53 \cdot \frac{\nu(N) 2^{\omega(N)}}{N} 
        + \frac{\sqrt{6}}{3} \cdot \frac{\nu(N)}{\sqrt{N}}
        + (67846+1) \cdot \frac{\nu(N)}{N} \\
        &\leq \frac53 \cdot \frac{21.234\cdot N^{1/16} \cdot 4.862\cdot N^{1/4}}{N} 
        + \frac{\sqrt{6}}{3} \cdot \frac{21.234\cdot N^{1/16}}{\sqrt{N}}
        + 67847 \cdot \frac{21.234\cdot N^{1/16}}{N},
\end{align}
which will be $< \frac{1}{12}$ for all $N \geq 58260767$.
This means that for $N \geq 58260767$, we will have $\dim S_2^\nw(\Gamma_0(N)) > 67846$. Then checking all the $N < 58260767$ by computer shows that we never have $\dim S_2^\nw(\Gamma_0(N)) = 67846$; see \cite{ross-code} for the code. 

We observe here that this is in the trivial character case, so the tools in Martin's paper (e.g. \cite[Proposition 26]{martin}) are sufficient to prove this result. 
It appears the main difference is just that we are using slightly better bounds and computed up to larger values of $N$.

\section{A character equidistribution property} \label{sec:char-equid}

We also obtain another result for free from the explicit dimension formula given in Theorem \ref{thm:explicit-dim-formula}.  In \cite{kimball}, Kimball Martin considered the decomposition of $S_k^\nw(\Gamma_0(N))$ into $S_k^\nw(\Gamma_0(N)) = S_k^{\nw,+}(\Gamma_0(N)) \oplus S_k^{\nw,-}(\Gamma_0(N))$. He then used an explicit dimension formula for $S_k^{\nw,\pm}(\Gamma_0(N))$ to show the asymptotic equidistribution of the two terms in this direct sum \cite[Corollary 2.3]{kimball}.

In our case, recall that $S_k^\nw(\Gamma_1(N))$ can be decomposed as $S_k^\nw(\Gamma_1(N)) = \bigoplus_{\chi} S_k^\nw(\Gamma_0(N),\chi)$ \cite[Theorem 7.3.4]{cohen-stromberg}. 
We can then similarly use our explicit dimension formula for $S_k^\nw(\Gamma_0(N),\chi)$ to ask about character equidistribution in this direct sum.

For comparison, first consider the full space $S_k(\Gamma_1(N)) = \bigoplus_{\chi} S_k(\Gamma_0(N),\chi)$. Then in the notation of Lemma \ref{lem:dim-linear-comb-mult},
\begin{align}
    \dim S_k(\Gamma_0(N),\chi) = \frac{k-1}{12} \psi(N) - c_3(k) \rho(N) - c_4(k) \rho'(N) - \frac{1}{2} \sigma(N) + c_0(k,\chi).
\end{align}
Here, we have directly from Formulas \ref{form:beta-convolute-psif} - \ref{form:beta-conv-rhopm} that
\begin{equation} \label{eqn:fullspace-terms-asymptotics}
\begin{aligned}
    \psi(N) &\geq N \qquad\qquad 
    & \rho(N) &= O(N^\varepsilon) \\
    \rho'(N)  &= O(N^\varepsilon) 
    & \sigma(N) &= O(N^{1/2+\varepsilon})
\end{aligned}
\end{equation}
for any $\varepsilon > 0$.
Note that $\rho$, $\rho'$, and $\sigma$ depend on $\chi$, but the largest term $\psi$ does not. This immediately gives the following result.

{
\renewcommand{\thetheorem}{\ref{thm:fullspace-char-equid}}
\begin{theorem} 
    Consider $N \geq 1$, $k \geq 2$, and $\chi$ a Dirichlet character modulo $N$ with $\chi(-1) = (-1)^k$.
    Recall that $S_k(\Gamma_1(N)) = \bigoplus_{\chi} S_k(\Gamma_0(N),\chi)$.
    Then asymptotically in $N$, the terms in this direct sum are equidistributed by character. 
\end{theorem}
\addtocounter{theorem}{-1}
}

We note here precisely what ``asymptotic equidistribution by character" means.
Consider $N \geq 1$, $k \geq 2$, and $\chi$ a Dirichlet character modulo $N$ with $\chi(-1) = (-1)^k$.
Let $s(N,k,\chi)$ denote $\dim S_k(\Gamma_0(N),\chi)$, $s(N,k)$ denote $\dim S_k(\Gamma_1(N)) = \sum_{\chi} \dim S_k(\Gamma_0(N),\chi)$, and $n(N,k)$ denote the number of $\chi$ in this sum. Further restrict to $N,k$ where $s(N,k) \neq 0$. Then Theorem \ref{thm:fullspace-char-equid} is saying that for all $\varepsilon > 0$, 
\begin{align}
    \lrabs{\frac{s(N,k,\chi)}{s(N,k)} - \frac{1}{n(N,k)}} < \varepsilon \qquad  \text{for all } k,\chi,
\end{align}
for sufficiently large $N$.

One might guess that a similar equidistribution property holds for the newspace $S_k^\nw(\Gamma_1(N)) = \bigoplus_{\chi} S_k^\nw(\Gamma_0(N),\chi)$. 
However, this is wildly untrue; recall from Proposition \ref{prop:infinite-family-trivial-newspace} that $S_k^\nw(\Gamma_0(N),\chi)$ is trivial for any characters $\chi$ with $2 \mid f$ and $2 \midmid N/f$. Even for characters not in this infinite family, one can see from Theorem \ref{thm:explicit-dim-formula} that $\dim S_k^\nw(\Gamma_0(N),\chi)$ depends heavily on the conductor $f = \condf(\chi)$.

However, since the main term $\psi(f) \cdot \beta {*} \psi_f(N/f)$ of Theorem \ref{thm:explicit-dim-formula} depends only on the conductor of $\chi$ (i.e. not on $\chi$ itself), we still have an equidistribution property if we only consider characters of a given conductor. We have from \eqref{eqn:expdim-terms-asymptotics} that
\begin{equation} 
\begin{aligned}
    \psi(f) \cdot \psi_f(N/f) &= \Omega(N^{1-\varepsilon})  \qquad\qquad
    & \rho(f) \cdot \beta{*}\rho_f(N/f) &= O(N^\varepsilon) \\
    \rho'(f) \cdot \rho_f'(N/f) &= O(N^\varepsilon) 
    & 2^{\omega(f)} \cdot \sigma_f(N/f) &= O(\sqrt{N})
\end{aligned}
\end{equation}
for any $\varepsilon > 0$.
The following result then follows immediately from Theorem \ref{thm:explicit-dim-formula} (with the second claim using the facts that $c_3(k)=c_4(k)=0$ when $k \equiv1\mod12$, and that $2^{\omega(f)} \cdot \sigma_f(N/f)$ depends only on the conductor of $\chi$).

{
\renewcommand{\thetheorem}{\ref{thm:newspace-char-equid}}
\begin{theorem} 
    Consider $N \geq 1$, $k \geq 2$, and $\chi$ a Dirichlet character modulo $N$ with $\chi(-1) = (-1)^k$.
    Recall that $S_k^\nw(\Gamma_1(N)) = \bigoplus_{\chi} S_k^\nw(\Gamma_0(N),\chi)$.
    Then asymptotically in $N$, the terms in this direct sum corresponding to characters of a given conductor are equidistributed by character. 
    Furthermore, if $k \equiv 1 \mod 12$, then for all $N$, the terms corresponding to characters of a given conductor are exactly equidistributed by character. 
\end{theorem}
\addtocounter{theorem}{-1}
}

More precisely, let $s'(N,k,\chi)$ denote $\dim S_k^\nw(\Gamma_0(N),\chi)$, $n_f'(N,k)$ denote the number of $\chi$  with conductor $f$, and $s_f'(N,k)$ denote the sum of $\dim S_k^\nw(\Gamma_0(N),\chi)$ over the $n_f'(N,k)$ characters $\chi$ of conductor $f$. Further restrict to $N,k,\chi$ where $s_{f}'(N,k) \neq 0$. Then Theorem \ref{thm:newspace-char-equid} is saying that for all $\varepsilon > 0$, 
    \begin{align}
        \lrabs{\frac{s'(N,k,\chi)}{s_{f}'(N,k)} - \frac{1}{n_{f}'(N,k)}} < \varepsilon \qquad  \text{for all } k,\chi,
    \end{align}
    for sufficiently large $N$.
    Furthermore,  
    \begin{align}
        \frac{s'(N,k,\chi)}{s_{f}'(N,k)} = \frac{1}{n_{f}'(N,k)} \qquad  \text{for all } k,\chi \text{ where } k \equiv 1 \!\!\!\!\mod 12,
    \end{align}
    for all $N$.

\section{A decomposition of certain \texorpdfstring{$S_k(\Gamma_0(N),\chi)$}{S\_k(Gamma0(N),chi)}} \label{sec:decomposition}

In Proposition \ref{prop:infinite-family-trivial-newspace}, we showed that if $2 \mid f$ and $2 \midmid N/f$, then $S_k^\nw(\Gamma_0(N),\chi)$ is trivial. However, one might ask for an intuitive reason why this is the case. In particular, since all of $S_k(\Gamma_0(N),\chi)$ is old, one could ask where each of the forms in $S_k(\Gamma_0(N),\chi)$ is coming from. By Atkin-Lehner-Li theory (e.g. \cite[Corollary 13.3.6]{cohen-stromberg}), we have that $S_k(\Gamma_0(N/2),\chi) \oplus S_k(\Gamma_0(N/2),\chi) |_k \pttmatrix{2}{0}{0}{1}  \subseteq S_k(\Gamma_0(N),\chi)$. And one can verify directly by Formula \ref{form:fullspace-dim} that in this case, we have $\dim S_k(\Gamma_0(N),\chi) = 2 \cdot \dim S_k(\Gamma_0(N/2),\chi)$. This means that in fact $S_k(\Gamma_0(N),\chi) = S_k(\Gamma_0(N/2),\chi) \oplus S_k(\Gamma_0(N/2),\chi) |_k \pttmatrix{2}{0}{0}{1}$. This yields the following result.
\begin{proposition} \label{prop:intuition}
    Let $N \geq 1$, $k \geq 2$, and $\chi$ be a Dirichlet character modulo $N$ of conductor $f=\condf(\chi)$ such that $\chi(-1) = (-1)^k$. Also assume that $2 \mid f$ and $2 \midmid N/f$. Then each $g \in S_k(\Gamma_0(N),\chi)$ can be decomposed uniquely as $g = g_1 + g_2 |_k \pttmatrix{2}{0}{0}{1}$ where $g_1,g_2\in S_k(\Gamma_0(N/2),\chi)$.
\end{proposition}
This proposition provides some additional understanding as to why $S_k^\nw(\Gamma_0(N),\chi)$ is trivial when $2 \mid f$ and $2 \midmid N/f$. However, one might still ask if there is a nice way to find the decomposition promised in Proposition \ref{prop:intuition}. Of course, for specific $N$, $k$, and $\chi$, one can calculate bases for each of the spaces involved and find the decomposition computationally. However, we surmise that there might be a more natural way to find this decomposition.


\section{Appendix} \label{sec:appendix}

Throughout this appendix, let $N \geq 1$, $k \geq 2$, and $\chi$ be a Dirichlet character modulo $N$ with conductor $f$ such that $\chi(-1) = (-1)^k$. 

First, we give a proof of the claimed formulas for $\sigma$, $\sigma_f$, and $\beta{*}\sigma_f$.

{
\renewcommand{\theformula}{\ref{form:beta-conv-sigma}}

\begin{formula} 
    For $p$ prime, let $\alpha := v_p(f)$.
    Then assuming $r \geq \alpha$, $\sigma$ is the multiplicative function defined for $r \geq 1$ by
    \begin{align}
        \sigma(p^r) &= \begin{cases}
            2 p^{r-\alpha} \qquad\qquad\ \ \ &\text{ if } \alpha \leq r < 2\alpha \\
            2 p^{(r-1)/2}  &\text{ if } r \geq 2\alpha, \ r \text{ odd} \\
            \lrp{1+\frac1p} p^{r/2} &\text{ if } r \geq 2\alpha, \ r \text{ even}. 
        \end{cases}  \label{eqn:app-sigma-computation} 
    \end{align}
    Furthermore, $\sigma_f$ and $\beta {*} \sigma_f$ are the multiplicative functions defined for $r \geq 1$ by
    \begin{align}
        \sigma_f(p^r) &=  \begin{cases}
            2 p^{(r-1)/2}  & \text{if $r$ is odd} \\
            \lrp{1+\frac1p}p^{r/2} \qquad\ \ & \text{if $r$ is even} 
        \end{cases} \, \qquad\qquad\qquad\qquad   \text{for } \alpha=0 \label{eqn:app-sigmaf-a-0} \\
        \sigma_f(p^r) &= \begin{cases}
            p^r & \text{if } r<\alpha \\
            p^{(r+\alpha-1)/2} & \text{if } r \geq \alpha,\ r+\alpha \text{ odd} \\
            \frac12 \lrp{1+\frac1p} p^{(r+\alpha)/2} &\text{if } r \geq \alpha,\ r+\alpha \text{ even} 
        \end{cases} \qquad\qquad\ \! \text{for } \alpha\geq 1  \label{eqn:app-sigmaf-a-ge1} \\
        \beta {*} \sigma_f(p^r) &=
        \begin{cases}
            0 & \text{ if $r$ is odd} \\
            p-2 & \text{ if $r=2$} \\
            \lrp{1 - \frac2p + \frac{1}{p^2}} p^{r/2} & \text{ if $r \geq 4$ even}
        \end{cases} 
        \qquad\qquad\qquad \ \ \ \ 
        \text{for } \alpha = 0   \\
        \beta {*} \sigma_f(p) &= 
        \begin{cases}
            \frac12\lrp{p-3} \qquad\qquad\qquad & \text{if } \alpha=1 \\
            p-2 & \text{if } \alpha \geq 2 
        \end{cases} 
        \qquad\qquad\qquad\qquad\qquad\!\! \text{for } \alpha \geq 1   \\
        \beta {*} \sigma_f(p^r) &=
        \begin{cases}
            0 & \text{if } r\geq\alpha+1,\, r+\alpha \text{ odd} \\ 
            \frac12 \lrp{1-\frac2p+\frac{1}{p^2}} p^{(r+\alpha)/2} & \text{if } r \geq \alpha+2,\, r+\alpha \text{ even} \\
            \frac12 \lrp{1-\frac3p+\frac{2}{p^2}} p^{r} & \text{if } r=\alpha  \\
            \lrp{1-\frac2p+\frac{1}{p^2}} p^r & \text{if } r < \alpha.
        \end{cases}
        \ \  \text{for } \alpha \geq 1,\, r \geq 2 
    \end{align}
\end{formula}

\addtocounter{formula}{-1}
}

\begin{proof}
    First, if $\alpha \leq r < 2\alpha$, this implies that we have $r - \alpha < r/2$. So by \eqref{eqn:sigma-def}, this gives
    \begin{align}
        \sigma(p^r) 
        &= \sum_{
                \substack{d \mid p^r \\   \gcd(d,p^r/d) \mid p^{r-\alpha} }
            }\mspace{-20mu} \phi(\gcd(d, p^r / d)) \\
        &= \sum_{
            \substack{0 \leq s \leq r \\  \min(s,r-s)  \leq r-\alpha}
        }\mspace{-20mu} \phi(p^{\min(s,r-s)}) \\
        &= 2 \sum_{
            0 \leq s \leq r-\alpha
        }\mspace{-5mu} \phi(p^s) \\
        &= 2 p^{r-\alpha}, 
    \end{align}
    as claimed. Here, we used the well-known identity $\sum_{d \mid N} \phi(d) = N$.
 
    Second, if $r \geq 2\alpha$ with $r$ odd, then
    \begin{align}
        \sigma(p^r) 
        &= \sum_{
            \substack{0 \leq s \leq r \\  \min(s,r-s)  \leq r-\alpha}
        }\mspace{-20mu} \phi(p^{\min(s,r-s)}) \\
        &= \sum_{
            0 \leq s \leq r 
        } \phi(p^{\min(s,r-s)}) \\
        &= 2 \sum_{
            0 \leq s \leq (r-1)/2
        }\mspace{-5mu} \phi(p^s) \\
        &= 2 p^{(r-1)/2}, 
    \end{align}
    as claimed.

    Third, if $r \geq 2\alpha$ with $r$ even, then
    \begin{align}
        \sigma(p^r) 
        &= \sum_{
            \substack{0 \leq s \leq r \\  \min(s,r-s)  \leq r-\alpha}
        }\mspace{-20mu} \phi(p^{\min(s,r-s)}) \\
        &= \sum_{
            0 \leq s \leq r 
        } \phi(p^{\min(s,r-s)}) \\
        &= 2 \sum_{
            0 \leq s \leq r/2
        }\mspace{-5mu} \phi(p^s)  \ - \phi(p^{r/2})\\
        &= 2 p^{r/2} - \lrp{1 - \frac1p}p^{r/2} \\
        &= \lrp{1 + \frac1p}p^{r/2}.
    \end{align}

    Fourth and fifth, if $\alpha = 0$, then by \eqref{eqn:theta-sub-f}, we have $\sigma_f(p^r) = \sigma(p^r)$. So from \eqref{eqn:app-sigma-computation}, 
    $\sigma_f(p^r) = 2 p^{(r-1)/2}$ for $r$ odd, and 
    $\sigma_f(p^r) = \lrp{1+\frac1p} p^{r/2}$ for $r$ even.
    
    Sixth, seventh, and eighth, if $\alpha \geq 1$, then by \eqref{eqn:theta-sub-f} and \eqref{eqn:app-sigma-computation}, we have 
    \begin{align}
        \sigma_f(p^r) &= \frac{\sigma(p^{r+\alpha})}{\sigma(p^\alpha)} = \frac12 \sigma(p^{r+\alpha}) 
        = \begin{cases}
            p^r & \text{if } r<\alpha \\
            p^{(r+\alpha-1)/2} & \text{if } r \geq \alpha,\ r \text{ odd} \\
            \frac12 \lrp{1+\frac1p} p^{(r+\alpha)/2} &\text{if } r \geq \alpha,\ r \text{ even}.
        \end{cases} 
    \end{align}

    Ninth, tenth, and eleventh, if $\alpha = 0$, then by \eqref{eqn:app-sigmaf-a-0}, 
    \begin{align}
        \beta{*}\sigma_f(p) &= \sigma_f(p) - 2 = 0 \\
        \beta{*}\sigma_f(p^r) &= \sigma_f(p^r) - 2 \sigma_f(p^{r-1}) + \sigma_f(p^{r-2}) \\
        &= 2p^{(r-1)/2} - 2\lrp{1+\frac1p} p^{(r-1)/2} + 2p^{(r-3)/2} \\
        &= 0 
        \qquad\qquad\qquad\qquad\qquad \qquad\qquad
        \text{for } r \geq 3 \text{ odd} \\
        \beta{*}\sigma_f(p^2) &= \sigma_f(p^2) - 2 \sigma_f(p) + 1 = \lrp{p+1} - 2 \cdot 2 + 1 = p-2 \\
        \beta{*}\sigma_f(p^r) &= \sigma_f(p^r) - 2 \sigma_f(p^{r-1}) + \sigma_f(p^{r-2}) \\
        &= \lrp{1+\frac1p} p^{r/2} - 2\cdot 2p^{(r-2)/2} + \lrp{1+\frac1p} p^{(r-2)/2} \\
        &= \lrp{1-\frac{2}{p} + \frac{1}{p^2} } p^{r/2} 
        \qquad\qquad\qquad \ \ \ 
        \text{for } r \geq 4 \text{ even}.
    \end{align} 

    Twelfth and thirteenth, if $\alpha \geq 1$, then by \eqref{eqn:app-sigmaf-a-ge1}, 
    \begin{align}
        \beta {*} \sigma_f(p) &= \sigma_f(p) - 2 = \begin{cases}
            \frac12 \lrp{p - 3}  & \text{if } \alpha=1 \\
            p - 2  & \text{if } \alpha \geq 2.
        \end{cases}
    \end{align}

    Finally, fourteenth through seventeenth, assume $\alpha \geq 1$ and $r \geq 2$. 
    
    Fourteenth, if $r \geq \alpha+1$ with $r+\alpha$ odd, then observe that in every case of \eqref{eqn:app-sigmaf-a-ge1}, $\sigma_f(p^{r-2}) = p^{(r+\alpha-3)/2}$. Thus by \eqref{eqn:app-sigmaf-a-ge1},
    \begin{align}
        \beta {*} \sigma_f(p^r) &= \sigma_f(p^r) - 2 \sigma_f(p^{r-1}) + \sigma_f(p^{r-2}) \\
        &= p^{(r+\alpha-1)/2} - 2 \cdot \frac12 \lrp{1+\frac1p} p^{(r+\alpha-1)/2} + p^{(r+\alpha-3)/2} \\
        &= 0.
    \end{align}
    
    Fifteenth, if $r \geq \alpha+2$ with $r+\alpha$ even, then by \eqref{eqn:app-sigmaf-a-ge1},
    \begin{align}
        \beta {*} \sigma_f(p^r) &= \sigma_f(p^r) - 2 \sigma_f(p^{r-1}) + \sigma_f(p^{r-2}) \\
        &= \frac12 \lrp{1+\frac1p} p^{(r+\alpha)/2} - 2 \cdot p^{(r+\alpha-2)/2} + \frac12 \lrp{1+\frac1p} p^{r+\alpha-2} \\
        &= \frac12 \lrp{1 - \frac2p + \frac{1}{p^2}} p^{(r+\alpha)/2}.
    \end{align}

    Sixteenth, if $r = \alpha$, then by \eqref{eqn:app-sigmaf-a-ge1},
    \begin{align}
        \beta {*} \sigma_f(p^r) &= \sigma_f(p^r) - 2 \sigma_f(p^{r-1}) + \sigma_f(p^{r-2}) \\
        &= \frac12 \lrp{1+\frac1p} p^r - 2 p^{r-1} + p^{r-2} \\
        &= \frac12 \lrp{1 - \frac3p + \frac{2}{p^2}} p^r.
    \end{align}
    Seventeenth, if $r < \alpha$, then by \eqref{eqn:app-sigmaf-a-ge1}, 
    \begin{align}
        \beta {*} \sigma_f(p^r) &= \sigma_f(p^r) - 2 \sigma_f(p^{r-1}) + \sigma_f(p^{r-2}) \\
        &= p^r - 2 p^{r-1} + p^{r-2} \\
        &=  \lrp{1 - \frac2p + \frac{1}{p^2} } p^r,
    \end{align}
    completing the proof.
\end{proof}

Next, we give a proof of the claimed formulas for $\rho$, $\rho_f$, and $\beta {*} \rho_f$.

{
\renewcommand{\theformula}{\ref{form:beta-conv-rho}}
\begin{formula}
    For $p$ prime, let $\alpha := v_p(f)$, and let $\chi_{p^\alpha}$ be defined as in Lemma \ref{lem:dim-linear-comb-mult}. If $\lrp{\frac{-3}{p}} = 1$, then let $u$ denote a square root of $-3$ modulo $p^\alpha$. Then assuming $r \geq \alpha$, $\rho$ is the multiplicative function defined for $r \geq 1$ by
    \begin{align}
        \rho(p^r) &= 
        \begin{cases}
            0  & \text{if } p=2 \\
            1  & \text{if } p=3;\ r=1 \\
            0  & \text{if } p=3;\ r\geq2 \\
            0  & \text{if }  p\neq2,3;\  \lrp{\frac{-3}{p}} = -1 \\
            2  & \text{if } p\neq2,3;\  \lrp{\frac{-3}{p}} = 1;\  \chi_{p^\alpha}\lrp{\frac{-1+u}{2}} = 1 \\
            -1  & \text{if } p\neq2,3;\  \lrp{\frac{-3}{p}} = 1;\  \chi_{p^\alpha}\lrp{\frac{-1+u}{2}} = \frac{-1 \pm \sqrt{-3}}{2}. \\
        \end{cases} \label{eqn:app-rho-computation}
    \end{align}
    Furthermore, if $\rho(f) \neq 0$, then $\rho_f$ and $\beta {*} \rho_f$ are the multiplicative functions defined for $r \geq 1$ by
    \begin{align}
        \rho_f(p^r) &= \begin{cases}
            1 & \text{ if } p=3;\ r=1;\ \alpha=0 \\
            1 & \text{ if }p\neq3;\ \alpha\geq1 \\
            2 & \text{ if } p\neq2,3;\ \lrp{\frac{-3}{p}}=1;\ \alpha=0 \\
            0 & \text{ otherwise}
        \end{cases}  \label{eqn:app-rhof-computation} \\
        \beta {*} \rho_f(p) &= \begin{cases}
            -1 & \text{ if } p=3;\ \alpha=0 \\
            -1 & \text{ if }  p\neq3;\ \alpha\geq1 \\
            0  & \text{ if }  p\neq2,3;\ \lrp{\frac{-3}{p}}=1;\ \alpha=0 \\
            -2 & \text{ otherwise}
        \end{cases}  \\
        \beta {*} \rho_f(p^2) &= \begin{cases}
            -1 & \text{ if } p=3;\ \alpha=0 \\
            0 & \text{ if } p\neq3;\ \alpha\geq1 \\
            -1 & \text{ if } p\neq2,3;\ \lrp{\frac{-3}{p}}=1;\ \alpha=0 \\
            1 & \text{ otherwise}
        \end{cases} \\
        \beta {*} \rho_f(p^r) &= \begin{cases}
            1 & \text{ if }  p=3;\ r=3;\ \alpha=0 \\
            0 & \text{ if }  r\geq3, \text{ otherwise}. 
        \end{cases} 
    \end{align}
\end{formula}
\addtocounter{formula}{-1}
}

\begin{proof}
    Recall that for all $r \geq \alpha$,
    \begin{align} \label{eqn:app-rho-formula-temp}
        \rho(p^r) = \sum_{
                \substack{x \mod p^r \\  x^2+x+1\equiv0 \mod p^r}
            }\mspace{-35mu} \chi_{p^\alpha}(x).
    \end{align}
 
    The first, second, and third claims follow immediately from applying  Lemma \ref{lem:congruence-solns} to \eqref{eqn:app-rho-formula-temp}.

    For the fourth, fifth, and sixth claims, assume  $p \neq 2,3$. Then note that since $p$ divides neither $2$, nor the discriminant $D=-3$ of $x^2+x+1$, we can use the quadratic formula here. If $\lrp{\frac{-3}{p}} = -1$, then a square root of $-3$ does not exist modulo $p^r$, and hence  $x^2+x+1 \equiv 0 \mod p^r$ has no solutions. So in this case $\rho(p^r) = 0$, verifying the fourth claim.  If $\lrp{\frac{-3}{p}} = 1$, then a square root $\hat{u}$ of $-3$ does exists modulo $p^r$, and  the two solutions of $x^2+x+1 \equiv 0 \mod p^r$ are $x \equiv (-1 \pm \hat{u})/2 \mod p^r$. Thus
    \begin{align}
        \rho(p^r) = \chi_{p^\alpha}\lrp{\frac{-1 + \hat{u}}{2}} + \chi_{p^\alpha}\lrp{\frac{-1 - \hat{u}}{2}} =
        \chi_{p^\alpha}\lrp{\frac{-1 + u}{2}} + \chi_{p^\alpha}\lrp{\frac{-1 - u}{2}}.
    \end{align}
    Here, the second equality comes from reducing modulo $p^\alpha$, and recalling that $u$ denotes a square root of $-3$ modulo $p^\alpha$.
    Now, observe that 
    \begin{align}
        \chi_{p^\alpha}\lrp{\frac{-1 \pm u}{2}}^3 &= \chi_{p^\alpha}\lrp{\frac{-1 \pm 3u - 3u^2 \pm u^3}{8}} = \chi_{p^\alpha}\lrp{\frac{-1 \pm 3u + 9 \mp 3u}{8}} = \chi_{p^\alpha}\lrp{1} = 1,
    \end{align}
    which means that $\chi_{p^\alpha}\lrp{\frac{-1 \pm u}{2}}$ are third roots of unity. Additionally,
    \begin{align}
        \chi_{p^\alpha}\lrp{\frac{-1 + u}{2}} \cdot \chi_{p^\alpha}\lrp{\frac{-1 - u}{2}} = \chi_{p^\alpha}\lrp{\frac{1 - u^2}{4}} = \chi_{p^\alpha}\lrp{\frac{4}{4}} = 1,
    \end{align}
    which means that $\chi_{p^\alpha}\lrp{\frac{-1 \pm u}{2}}$ are complex conjugates. The only way for this to happen is to have
    \begin{align}
        \chi_{p^\alpha}\lrp{\frac{-1 \pm u}{2}} = 1
         \qquad \text{or} \qquad 
         \chi_{p^\alpha}\lrp{\frac{-1 \pm u}{2}} = \frac{-1 \pm \sqrt{-3}}{2}.
    \end{align}
    This then yields
    \begin{align}
        \rho(p^r) &=
        \chi_{p^\alpha}\lrp{\frac{-1 + u}{2}} + \chi_{p^\alpha}\lrp{\frac{-1 - u}{2}} = \begin{cases}
            2 & \text{if }  \chi_{p^\alpha}\lrp{\frac{-1 + u}{2}} = 1 \\
            -1 & \text{if }  \chi_{p^\alpha}\lrp{\frac{-1 + u}{2}} = \frac{-1 \pm \sqrt{-3}}{2},
        \end{cases}
    \end{align}
    verifying the fifth and sixth claims.

    For the seventh through tenth claims, we have from \eqref{eqn:theta-sub-f} and \eqref{eqn:app-rho-computation} that 
    \begin{align}
        \rho_f(3^r) &= \frac{\rho(3^{r+\alpha})}{\rho(3^\alpha)} = \begin{cases}
            1 & \text{if } r=1,\,\alpha=0 \\
            0 & \text{if } r \geq 2 \text{ or } \alpha \geq 1.
        \end{cases}
    \end{align}
    Otherwise, assume $p \neq 3$. Note in each case of \eqref{eqn:app-rho-computation}, assuming $r\geq1$, the value of $\rho(p^r)$ does not depend on $r$. Thus for $\alpha \geq 1$, we have by \eqref{eqn:theta-sub-f} that $\rho_f(p^r) = \rho(p^{r+\alpha}) / \rho(p^{\alpha}) = 1$.
    For $\alpha=0$, note that $\chi_{p^\alpha}$ is the trivial character, so by \eqref{eqn:theta-sub-f} and \eqref{eqn:app-rho-computation},
    \begin{align}
        \rho_f(p^r) = \rho(p^r) = \begin{cases}
            2 & \text{if } p\neq2,3; \lrp{\frac{-3}{p}}=1 \\
            0 & \text{otherwise},
        \end{cases}
    \end{align}
    verifying the seventh through tenth claims.

    For the eleventh through fourteenth claims, applying each of the cases from \eqref{eqn:app-rhof-computation} to Lemma \ref{lem:beta-convolution} yields
    \begin{align}
        \beta {*} \rho_f(p) &= \rho_f(p) - 2 = \begin{cases}
            -1 & \text{ if } p=3;\ \alpha=0 \\
            -1 & \text{ if }  p\neq3;\ \alpha\geq1 \\
            0  & \text{ if }  p\neq2,3;\ \lrp{\frac{-3}{p}}=1;\ \alpha=0 \\
            -2 & \text{ otherwise},
        \end{cases}
    \end{align}
    as desired.
    
    For the fifteenth through eighteenth claims, applying each of the cases from \eqref{eqn:app-rhof-computation} to Lemma \ref{lem:beta-convolution} yields
    \begin{align}
        \beta {*} \rho_f(p^2) &= \rho_f(p^2) - 2 \rho_f(p) + 1 = 
        \begin{cases}
            -1 & \text{ if } p=3;\ \alpha=0 \\
            0 & \text{ if } p\neq3;\ \alpha\geq1 \\
            -1 & \text{ if } p\neq2,3;\ \lrp{\frac{-3}{p}}=1;\ \alpha=0 \\
            1 & \text{ otherwise},
        \end{cases}
    \end{align}
    as desired.

    For the nineteenth and twentieth claims, observe from \eqref{eqn:app-rhof-computation} that when $r\geq3$,  we have $\rho_f(p^r) = \rho_f(p^{r-1}) = \rho_f(p^{r-2})$ in every case except for $p=3,\,r=3,\,\alpha=0$. Thus
    \begin{align}
        \beta {*} \rho_f(p^r) &= \rho_f(p^r) - 2 \rho_f(p^{r-1}) + \rho_f(p^{r-2}) = \begin{cases}
            1 & \text{if } p=3;\ r=3;\ \alpha=0 \\
            0 & \text{if $r\geq3$, otherwise},
        \end{cases}
    \end{align}
    completing the proof.
\end{proof}

Finally, we give a proof of the claimed formulas for $\rho'$, $\rho_f'$, and $\beta {*} \rho_f'$.

{
\renewcommand{\theformula}{\ref{form:beta-conv-rhopm}}
\begin{formula} 
    For $p$ prime, let $\alpha := v_p(f)$, and let $\chi_{p^\alpha}$ be defined as in Lemma \ref{lem:dim-linear-comb-mult}. If $\lrp{\frac{-1}{p}} = 1$, then let $u'$ denote a square root of $-1$ modulo $p^\alpha$. Then assuming $r \geq \alpha$, $\rho'$ is the multiplicative function defined for $r \geq 1$ by
    \begin{align}
        \rho'(p^r) &= 
        \begin{cases}
            1  & \text{if } p=2;\ r=1 \\
            0  & \text{if } p=2;\ r\geq2 \\
            0  & \text{if }  p\neq2;\  \lrp{\frac{-1}{p}} = -1 \\
            2 & \text{if } p\neq2;\  \lrp{\frac{-1}{p}} = 1;\ \chi_{p^\alpha}\lrp{u'} = 1 \\
            -2 & \text{if } p\neq2;\  \lrp{\frac{-1}{p}} = 1;\ \chi_{p^\alpha}\lrp{u'} = -1 \\
            0 & \text{if } p\neq2;\  \lrp{\frac{-1}{p}} = 1;\  \chi_{p^\alpha}\lrp{u'} = \pm i.
        \end{cases} \label{eqn:app-rhopm-computation}
    \end{align}
    Furthermore, if $\rho'(f) \neq 0$, then $\rho_f'$ and $\beta {*} \rho_f'$ are the multiplicative functions defined for $r \geq 1$ by
    \begin{align}
        \rho_f'(p^r) &= \begin{cases}
            1 & \text{ if } p=2;\ r=1;\ \alpha=0 \\
            1 & \text{ if }p\neq2;\ \alpha\geq1 \\
            2 & \text{ if } p\neq2;\ \lrp{\frac{-1}{p}}=1;\ \alpha=0 \\
            0 & \text{otherwise}
        \end{cases}  \label{eqn:app-rhopmf-computation} \\
        \beta {*} \rho_f'(p) &= \begin{cases}
            -1 & \text{ if } p=2;\ \alpha=0 \\
            -1 & \text{ if }  p\neq2;\ \alpha\geq1 \\
            0  & \text{ if }  p\neq2;\ \lrp{\frac{-1}{p}}=1;\ \alpha=0 \\
            -2 & \text{ otherwise}
        \end{cases}  \\
        \beta {*} \rho_f'(p^2) &= \begin{cases}
            -1 & \text{ if } p=2;\ \alpha=0 \\
            0 & \text{ if } p\neq2;\ \alpha\geq1 \\
            -1 & \text{ if } p\neq2;\ \lrp{\frac{-1}{p}}=1;\ \alpha=0 \\
            1 & \text{ otherwise}.
        \end{cases}  \\
        \beta {*} \rho_f'(p^r) &= \begin{cases}
            1 & \text{ if }  p=2;\ r=3;\ \alpha=0 \\
            0 & \text{ if }  r\geq3, \text{ otherwise}. 
        \end{cases} 
    \end{align}
\end{formula}
\begin{proof}
    Recall that for all $r \geq \alpha$,
    \begin{align} \label{eqn:app-rhopm-formula-temp}
        \rho'(p^r) = \sum_{
                \substack{x \mod p^r \\  x^2+1\equiv0 \mod p^r}
            }\mspace{-35mu} \chi_{p^\alpha}(x).
    \end{align}
 
    The first and second claims follow immediately from applying  Lemma \ref{lem:congruence-solns} to \eqref{eqn:app-rhopm-formula-temp}.

    For the third through sixth claims, assume $p \neq 2$. If $\lrp{\frac{-1}{p}} = -1$, then a square root of $-1$ does not exist modulo $p^r$, and hence  $x^2+1 \equiv 0 \mod p^r$ has no solutions. So in this case $\rho'(p^r) = 0$, verifying the third claim.  If $\lrp{\frac{-1}{p}} = 1$, then a square root $\hat{u}'$ of $-1$ does exist modulo $p^r$, and the two solutions of $x^2+1 \equiv 0 \mod p^r$ are $x \equiv \pm \hat{u}'  \mod p^r$. Thus
    \begin{align}
        \rho'(p^r) = \chi_{p^\alpha}\lrp{\hat{u}'} + \chi_{p^\alpha}\lrp{-\hat{u}'} = 
        \chi_{p^\alpha}\lrp{u'} + \chi_{p^\alpha}\lrp{-u'}.
    \end{align}
    Here, the second equality comes from reducing modulo $p^\alpha$, and recalling that $u'$ denotes a square root of $-1$ modulo $p^\alpha$.
    Now, observe that 
    \begin{align}
        \chi_{p^\alpha}\lrp{\pm u'}^4 &= \chi_{p^\alpha}\lrp{(\pm u')^4} = \chi_{p^\alpha}\lrp{1} = 1,
    \end{align}
    which means that $\chi_{p^\alpha}\lrp{\pm u'}$ are fourth roots of unity. Additionally,
    \begin{align}
        \chi_{p^\alpha}\lrp{u'} \cdot \chi_{p^\alpha}\lrp{-u'} = \chi_{p^\alpha}\lrp{1} = 1,
    \end{align}
    which means that $\chi_{p^\alpha}\lrp{\pm u'}$ are complex conjugates. The only way for this to happen is to have
    \begin{align}
        \chi_{p^\alpha}\lrp{\pm u'} = 1
         \qquad \text{or} \qquad \chi_{p^\alpha}\lrp{\pm u'} = -1
         \qquad \text{or} \qquad
         \chi_{p^\alpha}\lrp{\pm u'} = \pm i.
    \end{align}
    This then yields
    \begin{align}
        \rho'(p^r) &=
        \chi_{p^\alpha}\lrp{u'} + \chi_{p^\alpha}\lrp{-u'} = 
        \begin{cases}
            2 & \text{if }  \chi_{p^\alpha}\lrp{u'} = 1 \\
            -2 & \text{if }  \chi_{p^\alpha}\lrp{u'} = -1 \\
            0 & \text{if }  \chi_{p^\alpha}\lrp{u'} = \pm i,
        \end{cases}
    \end{align}
    verifying the fourth, fifth, and sixth claims.

    For the seventh through tenth claims, we have from \eqref{eqn:theta-sub-f} and \eqref{eqn:app-rhopm-computation} that 
    \begin{align}
        \rho_f'(2^r) = \frac{\rho'(2^{r+\alpha})}{\rho'(2^\alpha)} = \begin{cases}
            1 & \text{if } r=1,\,\alpha=0 \\
            0 & \text{if } r \geq 2 \text{ or } \alpha \geq 1.
        \end{cases}
    \end{align}
    Otherwise, assume $p \neq 2$. If $\lrp{\frac{-1}{p}} = -1$, then $\rho_f'(p^r) = 0$ from \eqref{eqn:app-rhopm-computation}. 
    If $\lrp{\frac{-1}{p}} = 1$ and $\alpha = 0$, note $\chi_{p^\alpha}$ is the trivial character, so by \eqref{eqn:theta-sub-f} and \eqref{eqn:app-rhopm-computation}, $\rho_f'(p^r) = \rho'(p^{r}) = 2$.
    If $\lrp{\frac{-1}{p}} = 1$ and $\alpha \geq 1$, then by \eqref{eqn:theta-sub-f} and \eqref{eqn:app-rhopm-computation}, $\rho_f'(p^r) = \rho'(p^{r+\alpha}) / \rho'(p^{\alpha}) = 1$. This verifies the seventh through tenth claims.

    For the eleventh through fourteenth claims, applying each of the cases from \eqref{eqn:app-rhopmf-computation} to Lemma \ref{lem:beta-convolution} yields
    \begin{align}
        \beta {*} \rho_f'(p) &= \rho_f(p) - 2 = \begin{cases}
            -1 & \text{ if } p=2;\ \alpha=0 \\
            -1 & \text{ if }  p\neq2;\ \alpha\geq1 \\
            0  & \text{ if }  p\neq2;\ \lrp{\frac{-1}{p}}=1;\ \alpha=0 \\
            -2 & \text{otherwise},
        \end{cases}
    \end{align}
    as desired.
    
    For the fifteenth through eighteenth claims, applying each of the cases from \eqref{eqn:app-rhopmf-computation} to Lemma \ref{lem:beta-convolution} yields
    \begin{align}
        \beta {*} \rho_f'(p^2) &= \rho_f'(p^2) - 2 \rho_f'(p) + 1 = 
        \begin{cases}
            -1 & \text{ if } p=2;\ \alpha=0 \\
            0 & \text{ if } p\neq2;\ \alpha\geq1 \\
            -1 & \text{ if } p\neq2;\ \lrp{\frac{-1}{p}}=1;\ \alpha=0 \\
            1 & \text{ otherwise},
        \end{cases}
    \end{align}
    as desired.

    For the nineteenth and twentieth claims, observe from \eqref{eqn:app-rhopmf-computation} that when $r\geq3$,  we have $\rho_f'(p^r) = \rho_f'(p^{r-1}) = \rho_f'(p^{r-2})$ in every case except for $p=2,\,r=3,\,\alpha=0$. Thus
    \begin{align}
        \beta {*} \rho_f'(p^r) &= \rho_f'(p^r) - 2 \rho_f'(p^{r-1}) + \rho_f'(p^{r-2}) = \begin{cases}
            1 & \text{ if } p=2;\ r=3;\ \alpha=0 \\
            0 & \text{ if $r\geq3$, otherwise},
        \end{cases}
    \end{align}
    completing the proof.
\end{proof}
\addtocounter{formula}{-1}
}

\section*{Acknowledgments}
We would like to thank the anonymous referee for his/her helpful comments.

\begin{flushright} S.D.G.\end{flushright}

\bibliographystyle{plain}
\bibliography{bibliography.bib}

\begin{thebibliography}{10}

\bibitem{choi2}
SoYoung Choi, Chang~Heon Kim, and Kyung~Seung Lee.
\newblock Arithmetic properties for the minus space of weakly holomorphic modular forms.
\newblock {\em J. Number Theory}, 196:306--339, 2019.

\bibitem{cohen-stromberg}
Henri Cohen and Fredrik Str\"{o}mberg.
\newblock {\em Modular forms}, volume 179 of {\em Graduate Studies in Mathematics}.
\newblock American Mathematical Society, Providence, RI, 2017.
\newblock A classical approach.

\bibitem{dalal}
Tarun Dalal.
\newblock A basis for the space of weakly holomorphic {D}rinfeld modular forms of level {$T$}.
\newblock {\em J. Number Theory}, 258:122--145, 2024.

\bibitem{hardy-wright}
G.~H. Hardy and E.~M. Wright.
\newblock {\em An introduction to the theory of numbers}.
\newblock Oxford University Press, Oxford, sixth edition, 2008.
\newblock Revised by D. R. Heath-Brown and J. H. Silverman, With a foreword by Andrew Wiles.

\bibitem{lang}
Serge Lang.
\newblock {\em Algebraic number theory}, volume 110 of {\em Graduate Texts in Mathematics}.
\newblock Springer-Verlag, New York, second edition, 1994.

\bibitem{lmfdb-conrey-label}
The {LMFDB Collaboration}.
\newblock The {L}-functions and modular forms database.
\newblock \url{https://www.lmfdb.org/knowledge/show/character.dirichlet.conrey}, 2024.

\bibitem{martin}
Greg Martin.
\newblock Dimensions of the spaces of cusp forms and newforms on {$\Gamma_0(N)$} and {$\Gamma_1(N)$}.
\newblock {\em J. Number Theory}, 112(2):298--331, 2005.

\bibitem{kimball}
Kimball Martin.
\newblock Refined dimensions of cusp forms, and equidistribution and bias of signs.
\newblock {\em J. Number Theory}, 188:1--17, 2018.

\bibitem{kimball2023}
Kimball Martin.
\newblock Root number bias for newforms.
\newblock {\em Proc. Amer. Math. Soc.}, 151(9):3721--3736, 2023.

\bibitem{ross-code}
Erick Ross.
\newblock Nontrivial newspace.
\newblock \url{https://github.com/eross156/nontrivial-newspace}, 2024.

\bibitem{roy}
Manami Roy, Ralf Schmidt, and Shaoyun Yi.
\newblock Dimension formulas for {S}iegel modular forms of level 4.
\newblock {\em Mathematika}, 69(3):795--840, 2023.
\newblock With an appendix by Cris Poor and David S. Yuen.

\bibitem{shemanske-treneer-walling}
T.~Shemanske, S.~Treneer, and L.~Walling.
\newblock Constructing simultaneous {H}ecke eigenforms.
\newblock {\em Int. J. Number Theory}, 6(5):1117--1137, 2010.

\bibitem{stein}
William Stein.
\newblock {\em Modular forms, a computational approach}, volume~79 of {\em Graduate Studies in Mathematics}.
\newblock American Mathematical Society, Providence, RI, 2007.
\newblock With an appendix by Paul E. Gunnells.

\bibitem{zhang-zhou}
Yichao Zhang and Yang Zhou.
\newblock Dimension formulas for modular form spaces with character for {F}ricke groups.
\newblock {\em Acta Arith.}, 206(4):291--311, 2022.

\end{thebibliography}

\end{document}